\documentclass[11pt]{amsart}
\usepackage[utf8]{inputenc}
\usepackage{fancyhdr}
\usepackage{amsmath}
\usepackage{amssymb}
\usepackage{graphicx}
\usepackage{caption}
\usepackage[colorlinks=true]{hyperref}
\hypersetup{urlcolor=blue, citecolor=red}
\usepackage{enumitem}

\newtheorem{theorem}{Theorem}[section]
\newtheorem{corollary}[theorem]{Corollary}
\newtheorem{lemma}[theorem]{Lemma}
\newtheorem{proposition}[theorem]{Proposition}
\theoremstyle{definition}
\newtheorem{definition}[theorem]{Definition}

\newtheorem{remark}[theorem]{Remark}

\newcommand*{\R}{\mathbb{R}}

\DeclareMathOperator{\conv}{conv}

\title[]{Differential inclusions and the chattering phenomenon}

\author[]{Martin Ivanov}
\address{Faculty of Mathematics and Informatics, Sofia University, 5 James Bourchier Blvd., 1164 Sofia, Bulgaria}
\email{mdobrinov@fmi.uni-sofia.bg}

\begin{document}

\maketitle

\section{Introduction}

\subsection{Formulation of the problem}

Let $h:\R^n \to \R$ be a smooth function. We partition the space based on the sign of the function $h$:
\begin{align*}
    \Omega_+&:=\{x \in \mathbb{R}^n:h(x)>0\},\\
    \Omega_-&:=\{x \in \mathbb{R}^n:h(x)<0\},\\
    \Omega_0&:=\{x \in \mathbb{R}^n:h(x)=0\}.
\end{align*}

Let $f, g:\R^n \to \R^n$ be two smooth vector fields. We define the multi-valued mapping $F: \R^n \rightrightarrows \R^n$ by
\begin{equation}\label{mapping}
    F(x) :=
    \begin{cases}
        \{f(x)\}, & x \in \Omega_+,\\
        \{f(x),g(x)\}, & x \in \Omega_0,\\
        \{g(x)\}, & x \in \Omega_-.
    \end{cases}
\end{equation}

We assume that the following condition holds.
\begin{enumerate}[label=(A\arabic*),start=0]
    \item\label{A0} $Dh(x) \ne \mathbf{0}$ for each $x \in \Omega_0$, where $Dh(x)$ denotes the gradient of $h$ at $x$.
\end{enumerate}

We study the existence of local solutions of the differential inclusion
\begin{equation*}
    \dot{x}\in F(x), \qquad x(0)=x_0,
\end{equation*}
for an arbitrary $x_0\in\mathbb{R}^n$. More precisely, we ask whether
there exist $T>0$ and an absolutely continuous function
$x:[0,T] \to \R^n$ such that $x(0)=x_0$ and
$\dot{x}(t) \in F(x(t))$ for almost every $t\in[0,T]$. We call each such function $x$ a trajectory of $F$ starting from $x_0$.

\subsection{Previous results}

The main difficulty in studying the existence of trajectories of $F$ arises when the starting point belongs to the set $\Omega_0$, on which $F$ is discontinuous.
Outside $\Omega_0$, $F$ is single-valued, and the differential inclusion locally reduces to an ordinary differential equation with a smooth right-hand side. Hence, the local existence and uniqueness of trajectories of $F$ starting outside $\Omega_0$ follow directly.
However, on $\Omega_0$, $F$ is upper semi-continuous but has nonconvex values, and thus the standard local existence results for differential inclusions (see, e.g., \cite{Deimling} and \cite{Filippov88}) cannot be applied.

The existence of trajectories of $F$ starting from points of $\Omega_0$ depends on how the function $h$ changes along the trajectories of $f$ and $g$, which can be studied using the Lie derivatives of $h$ along $f$ and $g$.

Filippov studied the problem considered here in \cite{Filippov60} (see also Chapter 2, Section 10.2 of \cite{Filippov88}). His results can be summarized in a single statement, formulated in terms of the first-order Lie derivatives of $h$ along $f$ and $g$ (which we denote by $L_f h$ and $L_g h$, respectively), as follows.
A trajectory of $F$ starting from a point $x_0 \in \Omega_0$ satisfying $L_f h(x_0) L_g h(x_0) \ne 0$ exists if and only if $L_f h(x_0) L_g h(x_0)>0$ or $L_g h(x_0)<0<L_f h(x_0)$. When $x_0 \in \Omega_0$ satisfies $L_f h(x_0)<0<L_g h(x_0)$, neither $f$ nor $g$ is tangent to $\Omega_0$ in a neighborhood of $x_0$, while $f$ and $g$ point toward $\Omega_0$ from $\Omega_+$ and $\Omega_-$, respectively, and thus there is no trajectory of $F$ starting from $x_0$.

For a point $x_0 \in \Omega_0$ satisfying $L_f h(x_0) L_g h(x_0)=0$, it remains an open problem whether there exists a trajectory of $F$ starting from $x_0$. As we will show below, depending on the setting, such a trajectory may or may not exist. A partial result was obtained in \cite{DieciLopez} by Dieci and Lopez, who proved the existence of smooth trajectories starting from certain points of $\Omega_0$ at which exactly one of the Lie derivatives $L_f h$ and $L_g h$ is zero.
Their analysis can, in fact, be reformulated in terms of the Lie derivatives of $h$ along $f$ and $g$ of order up to 3.

The main difficulty in further studying the existence of trajectories starting from points $x_0 \in \Omega_0$ satisfying $L_f h(x_0) L_g h(x_0)=0$ is the possible presence of infinitely many points of nondifferentiability along such a trajectory, accumulating at its starting point. For example, this may occur when a trajectory of $F$ passes from $\Omega_+$ to $\Omega_-$, and vice versa, at infinitely many times accumulating at the starting time, a phenomenon known as chattering (see Example 3.2 in \cite{DieciLopez}).

\subsection{Chattering}

\subsection{Main result and steps of the proof}\label{introduction}

We make the following assumptions.

\begin{enumerate}[label=(A\arabic*)]
    \item\label{A1} There exists a positive integer $k_0$ such that $L^k_f h(x) = L^k_g h(x)$ for each $k \in \{0,\dots,k_0-1\}$ and each $x \in \Omega_0$.

    \item\label{A2} For each $x \in \Omega_0$ satisfying $L^k_f h(x)=L^k_g h(x)=0$ for each $k \in \{0,\dots,k_0-1\}$, we have $$L^{k_0}_f h(x) L^{k_0}_g h(x) \ne 0.$$
\end{enumerate}

Here, $L^k_f h$ and $L^k_g h$ denote the $k$-th-order Lie derivatives of $h$ along $f$ and $g$, respectively (see Definition \ref{LieDerivativeDef}).



We define the following sets.

$$S_0:= \left\{x \in \Omega_0: L^k_f h(x)
=0 \text{ for each } k \in \{1, \dots, k_0-1\}\right\},$$
$$S_1:= S_0 \setminus \{x : L^{k_0}_f h(x)<0<L^{k_0}_g h(x)\},$$
$$S_2:= S_0 \cap \{x : L^{k_0}_f h(x)<0<L^{k_0}_g h(x)\},$$

We now state our main result.

\begin{theorem}\label{main}
    Let \ref{A0}, \ref{A1} and \ref{A2} hold.
    \begin{enumerate}[label=(\alph*)]
        \item If $k_0=2$, a trajectory of $F$ starting from $x_0$ exists if and only if $x_0 \notin S_2$.

        \item If $k_0 \ge 3$, for each $x_0 \in \R^n$, there exists a trajectory of $F$ starting from $x_0$.
    \end{enumerate}
\end{theorem}

\subsection{Structure of the paper}

\section{Preliminary results}

We first adopt the following convention. For an integer $m \ge 0$, a function is said to be of class $C^m$ on a closed interval $[a,b]$ if it is of class $C^m$ on $(a,b)$ and each of its derivatives up to order $m$ admits a continuous extension to $[a,b]$. A function is said to be of class $C^{\infty}$ (or smooth) on a closed interval $[a,b]$ if it is of class $C^m$ on $[a,b]$ for each integer $m \ge 0$.

We next recall the notion of a Lie derivative of the function $h$ along a smooth vector field, which will be used to study the existence and behavior of trajectories of $F$ starting from points in $\Omega_0$.

\begin{definition}[\cite{Isidori,Khalil,SchattlerLedzewicz}]\label{LieDerivativeDef}
    For a smooth vector field $Y$ on $\R^n$, the Lie derivative of $h$ along the vector field $Y$ at a point $x \in \R^n$ is defined by $$L_Y h(x) := D h(x) \, Y(x),$$ where $Dh(x)$ denotes the gradient of the function $h$ at $x$, written as a row vector.

    The higher-order Lie derivatives of $h$ along $Y$ at a point $x \in \R^n$ are defined recursively by 
    $$L^0_Y h(x) := h(x)$$
    and
    $$L^k_Y h(x) := L_Y(L^{k-1}_Y h)(x)$$ for each integer $k \ge 1$.
\end{definition}

The following observation relates the Lie derivatives of $h$ along a smooth vector field $Y$ to the derivatives of $h$ along trajectories of $Y$. The result follows directly by repeatedly applying the chain rule and using the fact that the trajectories of $Y$ are smooth.

\begin{proposition}\label{derivativesY}
    Let $Y$ be a smooth vector field on $\R^n$. If $\bar x$ is a trajectory of $Y$ defined on an interval $[0,\bar T]$, then the function $t \mapsto h(\bar x(t))$ is smooth on $[0,\bar T]$, and $$\frac{\mathrm{d}^k}{\mathrm{d}t^k} h(\bar x(t))=L_Y^k h(\bar x(t))$$ for each $k \ge 1$ and each $t \in [0,\bar T]$.
\end{proposition}

For trajectories of the multi-valued mapping $F$, the situation is different, since at each point $x \in \Omega_0$ both $f(x)$ and $g(x)$ are admissible velocities. However, under \ref{A1}, the Lie derivatives of the function $h$ along $f$ and $g$ up to order $k_0-1$ coincide on $\Omega_0$, which yields a result, similar to that of Proposition \ref{derivativesY}, for the derivatives of $h$ along trajectories of $F$ up to order $k_0-1$. 

In fact, we will prove such a result for the possibly broader class of trajectories of the convexified multi-valued mapping $\conv F : \R^n \rightrightarrows \R^n$, defined by 
\begin{equation*}
    \conv F(x) :=
    \begin{cases}
        \{f(x)\}, & x \in \Omega_+,\\
        \{\lambda f(x)+(1-\lambda)g(x) : \lambda \in [0,1]\}, & x \in \Omega_0,\\
        \{g(x)\}, & x \in \Omega_-,
    \end{cases}
\end{equation*}
since we will also need it later. The proof uses the following lemma.

\begin{lemma}\label{lemma2}
    Let $k_0 \ge 2$. For each $k \in \{1,\dots,k_0-1\}$ and each $x \in \Omega_0$ such that $L_f h(x)=0$ (or, equivalently, $L_g h(x)=0$, by \ref{A1}), we have $$L^k_f h(x)=L_f (L^{k-1}_g h)(x)=L_g (L^{k-1}_f h)(x)=L^k_g h(x).$$
\end{lemma}
\begin{proof}
    We fix $k \in \{1,\dots,k_0-1\}$ and $x_0 \in \Omega_0$ such that $L_f h(x_0)=0$. Since $L_f h(x_0)=Dh(x_0) f(x_0)=0$ and $Dh(x_0) \ne \mathbf{0}$ (see \ref{A0}), it follows that the vector $f(x_0)$ is tangent to the smooth manifold $\Omega_0$ at $x_0$. Thus, there exist $\varepsilon>0$ and a smooth curve $\gamma:(-\varepsilon,\varepsilon) \to \Omega_0$ such that $\gamma(0)=x_0$ and $\dot{\gamma}(0)=f(x_0)$ (see, e.g., Proposition 3.23 of \cite{Lee}).
    Since $\gamma(t) \in \Omega_0$ for each $t \in (-\varepsilon,\varepsilon)$ and $L^{k-1}_f h=L^{k-1}_g h$ on $\Omega_0$ (see \ref{A1}), we have $$L^{k-1}_f h(\gamma(t))=L^{k-1}_g h(\gamma(t))$$ for each $t \in (-\varepsilon,\varepsilon)$. Differentiating both sides of this equality and then setting $t=0$, we get $$D(L^{k-1}_f h)(x_0) f(x_0)=D(L^{k-1}_g h)(x_0) f(x_0),$$ that is, $L^k_f h(x_0)=L_f(L^{k-1}_g h)(x_0)$. Interchanging $f$ with $g$ in the above argument, we obtain $L^k_g h(x_0)=L_g(L^{k-1}_f h)(x_0)$. Since $L^k_f h(x_0)=L^k_g h(x_0)$ (see \ref{A1}), we have $L^k_f h(x_0)=L_f(L^{k-1}_g h)(x_0)=L_g(L^{k-1}_f h)(x_0)=L^k_g h(x_0)$. This completes the proof.
\end{proof}

For each $k \in \{0,\dots,k_0-1\}$, in view of \ref{A1}, we introduce the function $L^k h : \R^n \to \R$ defined by 
\begin{equation*}
    L^k h(x) :=
    \begin{cases}
        L^k_f h(x), & x \in \Omega_+,\\
        L^k_f h(x)=L^k_g h(x), & x \in \Omega_0,\\
        L^k_g h(x), & x \in \Omega_-.
    \end{cases}
\end{equation*}
We note that each function $L^k h$ is smooth on the open set $\Omega_+ \cup \Omega_-$ and locally Lipschitz on $\R^n$, because $L^k_f h$ and $L^k_g h$ are smooth on $\R^n$ and coincide on $\Omega_0$.

The following proposition is the analogue of Proposition \ref{derivativesY} for trajectories of $\conv F$.

\begin{proposition}\label{derivatives}
    Let $k_0 \ge 2$. Let $\bar x$ be a trajectory of $\conv F$ defined on an interval $[0,\bar T]$. Then the function $t \mapsto h(\bar x(t))$ is of class $C^{k_0-1}$ on $[0,\bar T]$, and $$\frac{\mathrm{d}^k}{\mathrm{d}t^k}h(\bar x(t))=L^k h(\bar x(t))$$ for each $k \in \{1,\dots,k_0-1\}$ and each $t \in [0,\bar T]$.
\end{proposition}
\begin{proof}
    We set $$I_+:=\{t \in [0,\bar T]:\bar x(t) \in \Omega_+\},$$ $$I_-:=\{t \in [0,\bar T]: \bar x(t) \in \Omega_-\},$$ $$I_0:=\{t \in [0,\bar T] : \bar x(t) \in \Omega_0\}.$$
    Since $\bar x$ is absolutely continuous, there exists a set $N \subset [0,\bar T]$ of Lebesgue measure zero such that, for each $t \in [0,\bar T] \setminus N$, $\bar x$ is differentiable at $t$ and $\dot{\bar x}(t) \in \conv F(\bar x(t))$.

    We set $\Phi(t):=h(\bar x(t))$ for each $t \in [0,\bar T]$. We will prove that $\Phi$ is of class $C^{k_0-1}$ on $[0,\bar T]$ and $\Phi^{(k)}(t)=L^k h(\bar x(t))$ for each $k \in \{1,\dots,k_0-1\}$. We proceed by induction on $k$.

    We begin with the base case $k=1$. Let $t_0 \in [0,\bar T] \setminus N$ be arbitrary. Then $\Phi$ is differentiable at $t_0$ because $h$ is smooth on $\R^n$ and $\bar x$ is differentiable at $t_0$, and $\Phi'(t_0)=Dh(\bar x(t_0))\dot{\bar x}(t_0)$.
    If $t_0 \in I_+$, since $\conv F$ coincides with $f$ on $\Omega_+$, we have $\dot{\bar x}(t_0) = f(\bar x(t_0))$. Thus, $$\Phi'(t_0)=Dh(\bar x(t_0)) f(\bar x(t_0))=L_f h(\bar x(t_0))=L^1 h(\bar x(t_0)).$$ Analogously, if $t_0 \in I_-$, then $\dot{\bar x}(t_0) = g(\bar x(t_0))$ and $$\Phi'(t_0)=Dh(\bar x(t_0)) g(\bar x(t_0))=L_g h(\bar x(t_0))=L^1 h(\bar x(t_0)).$$ If $t_0 \in I_0$, then $\dot{\bar x}(t_0) = \lambda f(\bar x(t_0))+(1-\lambda)g(\bar x(t_0))$ for some $\lambda \in [0,1]$. Thus,
    \begin{align*}
        \Phi'(t_0)&=Dh(\bar x(t_0)) (\lambda f(\bar x(t_0)) + (1-\lambda) g(\bar x(t_0)))\\
        &=\lambda L_f h(\bar x(t_0))+(1-\lambda)L_g h(\bar x(t_0)).
    \end{align*}
    By \ref{A1}, $L_f h=L_g h=L^1 h$ on $\Omega_0$. Hence, $\Phi'(t_0)=L^1 h(\bar x(t_0))$.
    
    Therefore, $\Phi'(t)=L^1 h(\bar x(t))$ for almost every $t \in [0,\bar T]$. Since $t \mapsto L^1 h(\bar x(t))$ is continuous on $[0,\bar T]$ and $\Phi$ is absolutely continuous on $[0,\bar T]$ (as a composition of a smooth function and an absolutely continuous function defined on a compact interval), the Newton--Leibniz theorem implies that the function $\Phi$ is differentiable everywhere on $[0,\bar T]$ and $\Phi'(t)=L^1 h(\bar x(t))$ for each $t \in [0,\bar T]$ (the derivatives at $t=0$ and $t=\bar T$ are one-sided). This completes the proof in the case where $k_0=2$.

    We now assume that $k_0 \ge 3$. We also assume that there exists $k \in \{1,\dots,k_0-2\}$ such that $\Phi$ is of class $C^k$ on $[0,\bar T]$ and $\Phi^{(k)}(t)=L^k h(\bar x(t))$ for each $t \in [0,\bar T]$. We will prove that $\Phi^{(k)}$ is differentiable on $[0,\bar T]$ and $\Phi^{(k+1)}(t)=L^{k+1} h(\bar x(t))$ for each $t \in [0,\bar T]$.

    Let $t_0 \in (0,\bar T) \setminus N$ be arbitrary. We consider three cases, according to which of the sets $I_+$, $I_-$ and $I_0$ contains $t_0$.

    If $t_0 \in I_+$, since $\Omega_+$ is open and $\bar x$ is continuous, there exists $\delta>0$ such that $\bar x$ lies in $\Omega_+$ on $(t_0-\delta,t_0+\delta)$. Thus, $\Phi^{(k)}(t)=L^k h(\bar x(t))=L^k_f h(\bar x(t))$ for each $t \in (t_0-\delta,t_0+\delta)$. Since $L^k_f h$ is smooth on $\R^n$ and $\bar x$ is differentiable at $t_0$, the function $\Phi^{(k)}$ is differentiable at $t_0$. Moreover,
    \begin{equation*}
        \Phi^{(k+1)}(t_0)=(L^k_f h \circ\bar x)'(t_0)=D(L^k_f h)(\bar x(t_0)) \dot{\bar x}(t_0).
    \end{equation*}
    Since $\conv F$ coincides with $f$ on $\Omega_+$, we have $\dot{\bar x}(t_0)=f(\bar x(t_0))$. Thus, $$\Phi^{(k+1)}(t_0)=D(L^k_f h)(\bar x(t_0)) f(\bar x(t_0))=L^{k+1}_f h(\bar x(t_0))=L^{k+1} h(\bar x(t_0)).$$
    
    If $t_0 \in I_-$, in a similar way, by replacing $f$ and $\Omega_+$ with $g$ and $\Omega_-$, respectively, we obtain
    $$\Phi^{(k+1)}(t_0)=L^{k+1}_g h(\bar x(t_0))=L^{k+1} h(\bar x(t_0)).$$

    It remains to consider the case where $t_0 \in I_0$. We assume in addition that $t_0$ is not an isolated point of $I_0$. Then there exists a sequence $\{t_n\}_{n=1}^{\infty} \subseteq I_0 \setminus \{t_0\}$ such that $\lim_{n \to \infty} t_n = t_0$. Since $\Phi(t_0)=0$ and $\Phi(t_n)=0$ for each $n \ge 1$, $$\Phi'(t_0)=\lim_{n \to \infty} \frac{\Phi(t_n)-\Phi(t_0)}{t_n-t_0}=0.$$ Thus, $\Phi'(t_0)=L^1 h(\bar x(t_0))=0$ and Lemma \ref{lemma2} implies
    \begin{equation}\label{identity2}
        L^{k+1}_f h(\bar x(t_0))=L_g (L^k_f h)(\bar x(t_0)).
    \end{equation}

    Since $t_0 \in I_0 \setminus N$, we also have 
    \begin{equation}\label{velocity}
        \dot{\bar x}(t_0) = \lambda f(\bar x(t_0))+(1-\lambda)g(\bar x(t_0))
    \end{equation}
    for some $\lambda \in [0,1]$.
    
    We will prove that
    \begin{equation}\label{limit0}
        \lim_{t \to t_0} \frac{L^k h(\bar x(t))-L^k h(\bar x(t_0))}{t-t_0}=L^{k+1} h(\bar x(t_0)).
    \end{equation}
    Then it will follow that $\Phi^{(k)}$ is differentiable at $t_0$ and $\Phi^{(k+1)}(t_0)=L^{k+1} h(\bar x(t_0))$.
    
    The function $t \mapsto L^k_f h(\bar x(t))$ is differentiable at $t_0$ because $L^k_f h$ is smooth on $\R^n$ and $\bar x$ is differentiable at $t_0$, and
    \begin{align*}
        \lim_{t \to t_0} \frac{L^k_f h(\bar x(t))-L^k_f h(\bar x(t_0))}{t-t_0}&=(L^k_f h \circ \bar x)'(t_0)=D(L^k_f h)(\bar x(t_0)) \dot{\bar x}(t_0)\\
        &\overset{(\ref{velocity})}{=}\lambda L^{k+1}_f h(\bar x (t_0))+(1-\lambda) L_g(L^k_f h)(\bar x(t_0))\\
        &\overset{(\ref{identity2})}{=}L^{k+1}_f h(\bar x(t_0))=L^{k+1} h(\bar x(t_0)).
    \end{align*}

    Analogously, by interchanging $f$ and $g$ in the above argument, we obtain $$\lim_{t \to t_0} \frac{L^k_g h(\bar x(t))-L^k_g h(\bar x(t_0))}{t-t_0}=L^{k+1} h(\bar x(t_0)).$$

    We now prove $(\ref{limit0})$. Let $\varepsilon>0$ be arbitrary. By the two limits above, there exists $\delta>0$ such that $$\left|\frac{L^k_f h(\bar x(t))-L^k_f h(\bar x(t_0))}{t-t_0}-L^{k+1} h(\bar x(t_0))\right| < \varepsilon$$ and $$\left|\frac{L^k_g h(\bar x(t))-L^k_g h(\bar x(t_0))}{t-t_0}-L^{k+1} h(\bar x(t_0))\right| < \varepsilon$$ for each $t \in J:=(t_0-\delta,t_0) \cup (t_0,t_0+\delta)$.

    Since $t_0 \in I_0$, we have $L^k h(\bar x(t_0))=L^k_f h(\bar x(t_0))=L^k_g h(\bar x(t_0))$. Also, for each $t \in J$, either $t \in I_+ \cup I_0$ and $L^k h(\bar x(t))=L^k_f h(\bar x(t))$, or $t \in I_-$ and $L^k h(\bar x(t))=L^k_g h(\bar x(t))$. Hence, for each $t \in J$, we have
    $$\left|\frac{L^k h(\bar x(t))-L^k h(\bar x(t_0))}{t-t_0}-L^{k+1} h(\bar x(t_0))\right|<\varepsilon,$$ which proves $(\ref{limit0})$. Consequently, $\Phi^{(k)}$ is differentiable at $t_0$ and $\Phi^{(k+1)}(t_0)=L^{k+1} h(\bar x(t_0))$.

    Combining all the cases, since the isolated points of $I_0$ are at most countably many, we obtain $\Phi^{(k+1)}(t)=L^{k+1} h(\bar x(t))$ for almost every $t \in [0,\bar T]$. Since $t \mapsto L^{k+1} h(\bar x(t))$ is continuous on $[0,\bar T]$ and $\Phi^{(k)}$ is absolutely continuous on $[0,\bar T]$ (as a composition of a locally Lipschitz function and an absolutely continuous function defined on a compact interval), the Newton--Leibniz theorem implies that $\Phi^{(k)}$ is differentiable everywhere on $[0,\bar T]$ and $\Phi^{(k+1)}(t)=L^{k+1} h(\bar x(t))$ for each $t \in [0,\bar T]$ (the derivatives at $t=0$ and $t=\bar T$ are one-sided). This completes the induction and the proof.
\end{proof}

\section{Trajectories starting outside the singular set}

In this section, we prove that, locally, for each $x_0 \in \R^n \setminus S_0$, the convexified mapping $\conv F$ has a unique trajectory through $x_0$, which is also a trajectory of $F$.

If $x_0 \in \Omega_+ \cup \Omega_-$, then there exists a neighborhood of $x_0$ on which $\conv F =F$ is single-valued and coincides with $f$ if $x_0 \in \Omega_+$ and with $g$ if $x_0 \in \Omega_-$.
Hence, for any such $x_0$, the claim follows immediately from the theory of ordinary differential equations.

We therefore focus on the case $x_0 \in \Omega_0 \setminus S_0$, where we use the Lie derivatives of $h$ along $f$ and $g$ to prove the claim.
For each $x \in \Omega_0 \setminus S_0$, we define
$$m(x):=\min\{k \in \{1,\dots,k_0-1\} : L^k h(x) \ne 0\}.$$
(Note that, by the definition of $S_0$, at each point of $\Omega_0 \setminus S_0$ at least one of the functions $L^1 h, L^2 h, \dots, L^{k_0-1} h$ is nonzero.)
For each $x_0 \in \Omega_0 \setminus S_0$, by Proposition \ref{derivatives}, the parity of $m(x_0)$ and the sign of $L^{m(x_0)} h(x_0)$ locally determine a unique trajectory of $\conv F$ through $x_0$.

We therefore partition $\Omega_0 \setminus S_0$ in the following way:
$$P^+_o:=\{x \in \Omega_0 \setminus S_0: m(x) \text{ is odd and } L^{m(x)} h(x)>0\},$$
$$P^-_o:=\{x \in \Omega_0 \setminus S_0: m(x) \text{ is odd and } L^{m(x)} h(x)<0\},$$
$$P^+_e:=\{x \in \Omega_0 \setminus S_0: m(x) \text{ is even and } L^{m(x)} h(x)>0\},$$
$$P^-_e:=\{x \in \Omega_0 \setminus S_0: m(x) \text{ is even and } L^{m(x)} h(x)<0\}.$$

We now prove the claim stated at the beginning of this section and describe the behavior of the trajectories in terms of the sets $\Omega_+$, $\Omega_-$, $P^\pm_o$, $P^\pm_e$.

\begin{proposition}\label{onOmega0}
    For each $x_0 \in \R^n \setminus S_0$, there exists $\bar T>0$ such that $\conv F$ has a unique trajectory $\bar x$ on $[-\bar T,\bar T]$ with $\bar x(0)=x_0$. Moreover, $\bar x$ is a trajectory of $F$ on $[-\bar T,\bar T]$ and:
    \begin{enumerate}[label=(\alph*),leftmargin=*]
        \item if $x_0 \in \widetilde\Omega_+:=\Omega_+ \cup P^+_e$, then $\bar x(t) \in \Omega_+$ for each $t \in [-\bar T,0) \cup (0,\bar T]$;

        \item if $x_0 \in \widetilde\Omega_-:=\Omega_- \cup P^-_e$, then $\bar x(t) \in \Omega_-$ for each $t \in [-\bar T,0) \cup (0,\bar T]$;

        \item if $x_0 \in P^+_o$, then $\bar x(t) \in \Omega_-$ for each $t \in [-\bar T,0)$ and $\bar x(t) \in \Omega_+$ for each $t \in (0,\bar T]$;

        \item if $x_0 \in P^-_o$, then $\bar x(t) \in \Omega_+$ for each $t \in [-\bar T,0)$ and $\bar x(t) \in \Omega_-$ for each $t \in (0,\bar T]$.
    \end{enumerate}
\end{proposition}
\begin{proof}
    We first prove case (a). We consider two subcases.
    
    \smallskip\noindent\textit{Case (a.1): $x_0 \in \Omega_+$.}\hypertarget{casea1}{}
    Since $h$ is continuous, the set $\Omega_+$ is open, and thus there exists an open neighborhood $V \subseteq \Omega_+$ of $x_0$ on which $F$ is single-valued and coincides with $f$. Since $f$ is smooth on $\R^n$ (and thus on $V$), there exist a positive $\bar T$ and a unique solution $\bar x$ of
    \begin{equation}\label{diffeq1}
        \dot{x}(t) = f(x(t)), \, x(0)=x_0,
    \end{equation}
    on $[-\bar T,\bar T]$ such that $\bar x(t) \in V$ for each $t \in [-\bar T,\bar T]$. It follows that $\bar x$ is also a trajectory of $F$ (and thus of $\conv F$) on $[-\bar T,\bar T]$.

    We next prove that $\bar x$ is the unique trajectory of $\conv F$ on $[-\bar T,\bar T]$ such that $\bar x(0)=x_0$. Let $\widetilde x$ be any trajectory of $\conv F$ defined on $[-\bar T,\bar T]$ such that $\widetilde{x}(0)=x_0$. We set $$\bar T_1:=\sup\{t \in (0,\bar T] : \widetilde x(\tau) \in V \text{ for each } \tau \in [0,t)\}.$$
    Then $\bar T_1 \le \bar T$ and $\widetilde x(t) \in V$ for each $t \in [0,\bar T_1)$.
    Since $\conv F$ coincides with $f$ on $V$, it follows that $\widetilde x$ is also a solution of $(\ref{diffeq1})$ on $[0,\bar T_1)$. Since $(\ref{diffeq1})$ has a unique solution on $[0,\bar T_1) \subseteq [0,\bar T]$, we have $\bar x(t)=\widetilde x(t)$ for each $t \in [0,\bar T_1)$. Also, by continuity, $\bar x(\bar T_1)=\widetilde x(\bar T_1)$. Suppose that $\bar T_1 < \bar T$. We have $\widetilde x(\bar T_1) \in \overline V \setminus V$; otherwise, this would contradict the definition of $\bar T_1$. Since $\bar x(\bar T_1)=\widetilde x(\bar T_1)$, we have $\bar x(\bar T_1) \in \overline V \setminus V$, which contradicts the fact that $\bar x$ lies in $V$ on $[0,\bar T]$. Therefore, $\bar T_1=\bar T$ and $\bar x(t)=\widetilde x(t)$ for each $t \in [0,\bar T]$. In the same way, we obtain $\bar x(t)=\widetilde x(t)$ for each $t \in [-\bar T,0]$. Thus, $\bar x$ is the unique trajectory of $F$ on $[-\bar T,\bar T]$ such that $\bar x(0)=x_0$.

    \smallskip\noindent\textit{Case (a.2): $x_0 \in P_e^+$.}\hypertarget{casea2}{}    
    We set $m_0:=m(x_0)$. Since $L^{m_0} h(x_0)=L^{m_0}_f h(x_0)>0$, there exists an open neighborhood $V$ of $x_0$ such that $L^{m_0} h(x) > 0$ and $L^{m_0}_f h(x) > 0$ for each $x \in V$.
    As in the proof of case (a.1), there exist a positive $\bar T$ and a unique solution $\bar x$ of $(\ref{diffeq1})$ on $[-\bar T,\bar T]$ such that $\bar x(t) \in V$ for each $t \in [-\bar T,\bar T]$. We will show that $\bar x$ is a trajectory of $F$ on $[-\bar T,\bar T]$. We set $\Phi(t):=h(\bar x(t))$ for each $t \in [-\bar T,\bar T]$. Proposition \ref{derivativesY} implies, in particular, that $\Phi$ is of class $C^{m_0}$ on $[-\bar T,\bar T]$ and $\Phi^{(k)}(t)=L^k_f h(\bar x(t))$ for each $k \in \{1,\dots,m_0\}$ and each $t \in [-\bar T,\bar T]$.

    We use the Taylor expansion of $\Phi$ around zero with the Lagrange form of the remainder: for each $t \in [-\bar T,\bar T]$ we have $$\Phi(t)=\sum_{k=0}^{m_0-1}\frac{\Phi^{(k)}(0)}{k!}t^k+\frac{\Phi^{(m_0)}(\xi(t))}{m_0!}t^{m_0},$$ where $\xi(t)$ is between $0$ and $t$.

    Since $\Phi^{(k)}(0)=L^k_f h(x_0)=L^k h(x_0)=0$ for each $k \in \{0,\dots,m_0-1\}$, $\Phi^{(m_0)}(\xi(t))=L^{m_0}_f h(\bar x(\xi(t)))>0$ for each $t \in [-\bar T,\bar T]$ (because $\bar x$ stays in $V$ on $[-\bar T,\bar T]$), and $m_0$ is even, we have $$\Phi(t) = \frac{\Phi^{(m_0)}(\xi(t))}{m_0!}t^{m_0} > 0$$ for each $t \in [-\bar T,0) \cup (0,\bar T]$. (Note that the fact that $m_0$ is even plays a role only on $[-\bar T,0)$ to ensure $t^{m_0}>0$ there.) Thus, $\bar x(t) \in \Omega_+$ for each $t \in [-\bar T,0) \cup (0,\bar T]$. Hence, $F(\bar x(t))=\{f(\bar x(t))\}$ for each $t \in [-\bar T,0) \cup (0,\bar T]$.   
    Since $\bar x$ is a trajectory of $f$ on $[-\bar T,\bar T]$, it follows that $\bar x$ is a trajectory of $F$ (and thus of $\conv F$) on $[-\bar T,\bar T]$.

    We next prove that $\bar x$ is the unique trajectory of $\conv F$ on $[-\bar T,\bar T]$ such that $\bar x(0)=x_0$.
    Let $\widetilde x$ be any trajectory of $\conv F$ on $[-\bar T,\bar T]$ such that $\widetilde x(0)=x_0$. 
    We set $$\bar T_1:=\sup\{t \in (0,\bar T] : \widetilde x(\tau) \in V \text{ for each } \tau \in [0,t)\}.$$
    Then $\bar T_1 \le \bar T$ and $\widetilde x(t) \in V$ for each $t \in [0,\bar T_1)$.

    We set $\Psi(t):=h(\widetilde x(t))$ for each $t \in [0,\bar T_1]$. Since $m_0 \le k_0-1$, Proposition \ref{derivatives} implies, in particular, that $\Psi$ is of class $C^{m_0}$ on $[0,\bar T_1]$ and $\Psi^{(k)}(t)=L^k h(\widetilde x(t))$ for each $k \in \{1,\dots,m_0\}$ and each $t \in [0,\bar T_1]$.

    We use the Taylor expansion of $\Psi$ on the right-hand side of zero with the Lagrange form of the remainder: for each $t \in (0,\bar T_1]$ we have $$\Psi(t)=\sum_{k=0}^{m_0-1}\frac{\Psi^{(k)}(0)}{k!}t^k+\frac{\Psi^{(m_0)}(\zeta(t))}{m_0!}t^{m_0},$$ where $\zeta(t) \in (0,t)$.

    Since $\Psi^{(k)}(0)=L^k h(x_0)=0$ for each $k \in \{0,\dots,m_0-1\}$ and $\Psi^{(m_0)}(\zeta(t))=L^{m_0} h(\widetilde x(\zeta(t)))>0$ for each $t \in [0,\bar T_1]$ (because $\widetilde x$ stays in $V$ on $[0,\bar T_1)$), we have $$\Psi(t) = \frac{\Psi^{(m_0)}(\zeta(t))}{m_0!}t^{m_0} > 0$$ for each $t \in (0,\bar T_1]$. Hence, $\widetilde x(t) \in \Omega_+$ for each $t \in (0,\bar T_1]$, and thus $\conv F(\widetilde x(t))=\{f(\widetilde x(t))\}$ for each $t \in (0,\bar T_1]$. Since $\widetilde x$ is a trajectory of $\conv F$ on $[0,\bar T_1]$, it follows that $\widetilde x$ is a trajectory of $f$ on $[0,\bar T_1]$. Since $\widetilde x(0)=\bar x(0)=x_0$ and $\bar x$ is the unique solution of $(\ref{diffeq1})$ on $[0,\bar T_1]$, we conclude that $\widetilde x(t)=\bar x(t)$ for each $t \in [0,\bar T_1]$. As in the proof of case (a.1), it follows that $\bar T_1 = \bar T$ and $\bar x(t)=\widetilde x(t)$ for each $t \in [0,\bar T]$. Analogously, we obtain $\bar x(t)=\widetilde x(t)$ for each $t \in [-\bar T,0]$. Thus, $\bar x$ is the unique trajectory of $\conv F$ on $[-\bar T,\bar T]$ such that $\bar x(0)=x_0$. This completes the proof of case (a).

    Case (b) follows by exactly the same argument as in case (a): replace $\Omega_+$, $P_e^+$, $f$ and ``$>0$'' with $\Omega_-$, $P_e^-$, $g$ and ``$<0$'', respectively.
    
    The proof of case (c) is similar to that of case (a.2). We briefly outline the differences.
    Let $x_0 \in P^+_{o}$. Then $L^{m(x_0)} h(x_0)=L^{m(x_0)}_f h(x_0)=L^{m(x_0)}_g h(x_0)>0$, and thus there exists an open neighborhood $V$ of $x_0$ such that $L^{m(x_0)}_f h(x) > 0$ and $L^{m(x_0)}_g h(x) > 0$ for each $x \in V$.
    As in the proof of case (a.2), there exist a positive $\bar T_1$ and a unique solution $\bar x_1$ of $(\ref{diffeq1})$ on $[0,\bar T_1]$ such that $\bar x_1(t) \in V$ for each $t \in [0,\bar T_1]$.

    We consider the differential equation
    \begin{equation}\label{diffeq2}
        \dot{x}(t) = g(x(t)), \, x(0)=x_0.
    \end{equation}
    Since $g$ is smooth on $\R^n$ (and thus on $V$), there exist a positive $\bar T_2$ and a unique solution $\bar x_2$ of $(\ref{diffeq2})$ on $[-\bar T_2,0]$ such that $\bar x_2(t) \in V$ for each $t \in [-\bar T_2,0]$.

    We set $\bar T:=\min\{\bar T_1,\bar T_2\}$. We define $\bar x:[-\bar T,\bar T] \to V$ by 
    \begin{equation*}
        \bar x(t) := \begin{cases}
            \bar x_2(t), &\text{if } t \in [-\bar T,0)\\
            \bar x_1(t), &\text{if } t \in [0,\bar T]\\
        \end{cases}.
    \end{equation*}

    As in the proof of case (a.2), using one-sided Taylor expansions of $t \mapsto h(\bar x(t))$ on each side of zero and the fact that $m(x_0)$ is odd, we can show that $\bar x(t) \in \Omega_+$ for each $t \in (0,\bar T]$ and $\bar x(t) \in \Omega_-$ for each $t \in [-\bar T,0)$. Thus, $F(\bar x(t))=\{f(\bar x(t))\}$ for each $t \in (0,\bar T]$ and $F(\bar x(t))=\{g(\bar x(t))\}$ for each $t \in [-\bar T,0)$. Since $\bar x$ is a trajectory of $f$ on $(0,\bar T]$ and of $g$ on $[-\bar T,0)$, it follows that $\bar x$ is a trajectory of $F$ (and thus of $\conv F$) on $[-\bar T,\bar T]$.

    The proof that $\bar x$ is the unique trajectory of $\conv F$ on $[-\bar T,\bar T]$ such that $\bar x(0)=x_0$ follows by exactly the same argument as in case (a.2).

    The proof of case (d) is analogous to that of case (c): interchange $f$ with $g$ and $\Omega_+$ with $\Omega_-$, and replace $P_o^+$ and ``$>0$'' with $P_o^-$ and ``$<0$'', respectively.
\end{proof}

\begin{remark}\label{one-sided}
    Let $\bar x$ be a trajectory of $F$ defined on an interval $[0,\bar T]$. Then Proposition \ref{onOmega0} implies that the sets $$I_+:=\{t \in [0,\bar T] : \bar x(t) \in \widetilde\Omega_+\} \quad\text{ and }\quad I_-:=\{t \in [0,\bar T] : \bar x(t) \in \widetilde\Omega_-\}$$ are open in $[0,\bar T]$ and that $\bar x$ is a trajectory of $f$ on $I_+$ and of $g$ on $I_-$. It also implies that the set $$I_0:=\{t \in (0,\bar T) : \bar x(t) \in P_o^+ \cup P_o^-\}$$ is discrete and that $\bar x$ passes from $\Omega_+$ to $\Omega_-$ or vice versa at each time in $I_0$.
\end{remark}

In view of Proposition \ref{onOmega0} and Remark \ref{one-sided}, we call the sets $P^+_o$ and $P^-_o$ {\it switching sets}. Given a trajectory $\bar x$ of $F$ defined on an interval $[0,\bar T]$, we call a time $t_0 \in (0,\bar T)$ a {\it switching time} of $\bar x$, and the point $\bar x(t_0)$ a {\it switching point} of $\bar x$, if $\bar x(t_0) \in P^+_o \cup P^-_o$.

As a corollary of Proposition \ref{onOmega0}, we show that each trajectory of $\conv F$ is also a trajectory of $F$ as long as it remains outside the singular set $S_0$. This fact will be used in studying the existence of trajectories of $F$ starting from points in $S_2$ when $k_0 \ge 3$.

\begin{corollary}\label{coincidence}
    Let $\bar x$ be a trajectory of $\conv F$ defined on an interval $[0, \bar T]$. Then $\bar x$ is a trajectory of $F$ on the open subset $W:=\{t \in [0,\bar T]: \bar x(t) \notin S_0\}$ of $[0,\bar T]$, that is, $\dot{\bar x}(t) \in F(\bar x(t))$ for almost every $t \in W$.
\end{corollary}
\begin{proof}
    For each $t \in W$, since $\bar x(t) \in \R^n \setminus S_0$, Proposition \ref{onOmega0} implies that there exists $\delta>0$ such that $\bar x$ is a trajectory of $F$ on $[t-\delta,t+\delta] \cap [0,\bar T]$.
    Since $W$ can be covered by at most countably many such intervals, it follows that $\bar x$ is a trajectory of $F$ on $W$. This completes the proof.
\end{proof}

\section{Trajectories starting from the singular set}\label{singularSection}

We first show that the singular set is the only part of the space near which chattering can occur. Later in this section, we show that chattering necessarily occurs near some points of the singular set, and we construct trajectories of $F$ with such behavior.

\begin{proposition}\label{accumulation}
    Let $\bar x$ be a trajectory of $F$ defined on an interval $[0,\bar T]$. Let $\{\tau_n\}_{n=1}^{\infty} \subset [0,\bar T]$ be a sequence of distinct switching times of $\bar x$ such that $\lim_{n \to \infty} \tau_n=\tau_0 \in [0,\bar T]$. Then $\bar x(\tau_0) \in S_0$.
\end{proposition}
\begin{proof}
    Assume the contrary: $\bar x(\tau_0) \notin S_0$. Then $\bar x(\tau_0) \in \widetilde\Omega_+ \cup \widetilde\Omega_- \cup P_o^+ \cup P_o^-$. Regardless of which of these four sets contains $\bar x(\tau_0)$, by Proposition \ref{onOmega0}, there exists $\delta_0>0$ such that $\bar x(t) \in \widetilde\Omega_+ \cup \widetilde\Omega_-$ for each $t \in I:=((\tau_0-\delta_0,\tau_0) \cup (\tau_0,\tau_0+\delta_0)) \cap [0,\bar T]$. Since $\{\tau_n\}_{n=1}^{\infty}$ converges to $\tau_0$, there exists a positive integer $n_0$ such that $\tau_{n_0} \in I$. However, by the definition of a switching time, $\bar x(\tau_{n_0}) \in P_o^+ \cup P_o^-$, which contradicts $\bar x(\tau_{n_0}) \in \widetilde\Omega_+ \cup \widetilde\Omega_-$. Hence, $\bar x(\tau_0) \in S_0$. This completes the proof.
\end{proof}

Recall that the singular set is divided into two subsets:
$$S_1:= S_0 \setminus \{x \in \R^n : L^{k_0}_f h(x)<0<L^{k_0}_g h(x)\},$$
$$S_2:= S_0 \cap \{x \in \R^n : L^{k_0}_f h(x)<0<L^{k_0}_g h(x)\}.$$

We first study the existence of trajectories of $F$ starting from points in $S_1$. By \ref{A2}, at each point $x_0 \in S_1$, either $L^{k_0}_g h(x_0)<0<L^{k_0}_f h(x_0)$ or $L^{k_0}_f h(x_0)L^{k_0}_g h(x_0)>0$. Thus, at each point $x_0 \in S_1$, at least one of $L^{k_0}_f h(x_0) > 0$ and $L^{k_0}_g h(x_0) < 0$ holds. Using this fact, we prove that, for each point of $S_1$, there exists a trajectory of $F$ with no switchings starting from that point.

\begin{proposition}
    If $x_0 \in S_1$ is such that $L_f^{k_0} h(x_0)>0$ ($L_g^{k_0} h(x_0)<0$, respectively), then there exist $\bar T>0$ and a trajectory $\bar x$ of $F$ defined on $[0,\bar T]$ such that $\bar x(0)=x_0$ and $\bar x(t) \in \Omega_+$ ($\bar x(t) \in \Omega_-$, respectively) for each $t \in (0,\bar T]$.
\end{proposition}
\begin{proof}
    The case where $x_0 \in S_1$ is such that $L_f^{k_0} h(x_0)>0$ can be proved in the same way as the existence of a trajectory of $F$ is proved in case \hyperlink{casea2}{(a.2)} of Proposition \ref{onOmega0}: it suffices to consider only $[0,\bar T]$ instead of $[-\bar T,\bar T]$, and to replace $m_0$ with $k_0$. The proof of the case where $x_0 \in S_1$ is such that $L_g^{k_0} h(x_0)<0$ is identical, with $\Omega_+$, $f$ and ``$>0$'' replaced with $\Omega_-$, $g$ and ``$<0$'', respectively.
\end{proof}

We next study the existence of trajectories of $F$ starting
from points in $S_2$. We first show that, if such a trajectory exists, then it must have infinitely many switching times accumulating at its starting time. In fact, we prove the following more general statement.

\begin{proposition}\label{chattering}
    Let $\bar x$ be a trajectory of $F$ defined on an interval $[0,\bar T]$. If $t_0 \in [0,\bar T)$ is such that $\bar x(t_0) \in S_2$, then $\bar x$ has infinitely many switching times in $(t_0,\bar T]$ accumulating at $t_0$. Moreover, if $k_0$ is even and $t_0 \in (0,\bar T]$, then $\bar x$ has infinitely many switching times in $[0,t_0)$ accumulating at $t_0$.
\end{proposition}
\begin{proof}   
    We first consider the case where $t_0 \in [0,\bar T)$ is such that $\bar x(t_0) \in S_2$.    
    Assume the contrary: there exists $\varepsilon >0$ such that the trajectory $\bar x$ has no switching times in the interval $(t_0,t_0+\varepsilon)$, that is, $\bar x(t) \notin P^+_o \cup P^-_o$ for each $t \in (t_0,t_0+\varepsilon)$.

    Since $\bar x(t_0) \in S_2$, we have $L^{k_0}_f h(\bar x(t_0)) < 0 < L^{k_0}_g h(\bar x(t_0))$. By continuity, after possibly shrinking $\varepsilon$, we assume that
    \begin{equation}\label{inequalityChattering}
        L^{k_0}_f h(\bar x(t)) < 0 < L^{k_0}_g h(\bar x(t))
    \end{equation}
    for each $t \in [t_0,t_0+\varepsilon]$.    

    Suppose first that $\bar x(t) \in S_2$ for each $t \in [t_0,t_0+\varepsilon]$. By the definition of $S_2$, we have $L^{k_0-1} h(\bar x(t))=0$ for each $t \in [t_0,t_0+\varepsilon]$. Since $L^{k_0-1} h=L^{k_0-1}_f h=L^{k_0-1}_g h$ on $\Omega_0$ (and thus on $S_2$), we have $$L^{k_0-1}_f h(\bar x(t))=L^{k_0-1}_g h(\bar x(t))=0$$ for each $t \in [t_0,t_0+\varepsilon]$. Since $L^{k_0-1}_f$ and $L^{k_0-1}_g$ are smooth on $\R^n$ and $\bar x$ is differentiable almost everywhere on $[t_0,t_0+\varepsilon]$ (because it is absolutely continuous), we have
    \begin{equation}\label{equality}
        DL^{k_0-1}_f h(\bar x(t)) \dot{\bar x}(t)=DL^{k_0-1}_g h(\bar x(t)) \dot{\bar x}(t)=0
    \end{equation}
    for almost every $t \in [t_0,t_0+\varepsilon]$.
    
    Since $\bar x$ is a trajectory of $F$, for almost every $t \in [t_0,t_0+\varepsilon]$, $\dot{\bar x}(t)$ equals $f(\bar x(t))$ or $g(\bar x(t))$. Combined with $(\ref{equality})$, this implies that, for almost every $t \in [t_0,t_0+\varepsilon]$, either $$L^{k_0}_f h(\bar x(t))=DL^{k_0-1}_f h(\bar x(t))f(\bar x(t))=0$$ or $$L^{k_0}_g h(\bar x(t))=DL^{k_0-1}_g h(\bar x(t))g(\bar x(t))=0.$$ This contradicts inequality $(\ref{inequalityChattering})$.

    Therefore, there exists $t_1 \in (t_0,t_0+\varepsilon]$ such that $\bar x(t_1) \notin S_2$. Since $\bar x(t_1) \notin P^+_o \cup P^-_o$ by assumption, $\bar x(t_1) \notin S_1$ by inequality $(\ref{inequalityChattering})$, and the family $\{\widetilde\Omega_+,\widetilde\Omega_-,P_o^+,P_o^-,S_1,S_2\}$ forms a partition of $\R^n$, it follows that $\bar x(t_1) \in \widetilde\Omega_+ \cup \widetilde\Omega_-$. Without loss of generality, we assume that $\bar x(t_1) \in \widetilde\Omega_+$. Thus, there exists $\delta>0$ such that $\bar x(t) \in \widetilde\Omega_+$ for each $t \in (t_1-\delta,t_1]$ (see Remark \ref{one-sided}).
    Let $$t_2:=\inf\{t \in [t_0,t_1) : \bar x(\tau) \in \widetilde\Omega_+ \text{ for each } \tau \in (t,t_1]\}.$$
    Then $t_0 \le t_2 \le t_1-\delta$, $\bar x(t) \in \widetilde\Omega_+$ for each $t \in (t_2,t_1]$, and $\bar x(t_2) \in S_2$. Indeed, $\bar x(t_2) \notin \widetilde\Omega_+ \cup \widetilde\Omega_-$; otherwise, this would contradict the definition of $t_2$. Also, $\bar x(t_2) \notin P_o^+ \cup P_o^-$ by assumption, and $\bar x(t_2) \notin S_1$ by inequality $(\ref{inequalityChattering})$.

    Since $\bar x(t) \in \widetilde\Omega_+$ for each $t \in (t_2,t_1]$, it follows that $\bar x$ is a trajectory of $f$ on $[t_2,t_1]$ (see Remark \ref{one-sided}).    
    We set $\Phi(t):=h(\bar x(t))$ for each $t \in [t_2,t_1]$. Proposition \ref{derivativesY} implies, in particular, that $\Phi$ is of class $C^{k_0}$ on $[t_2,t_1]$ and $\Phi^{(k)}(t)=L^k_f h(\bar x(t))$ for each $k \in \{1,\dots,k_0\}$ and each $t \in [t_2,t_1]$.

    We use the Taylor expansion of $\Phi$ on the right-hand side of $t_2$ with the Lagrange form of the remainder: for each $t \in (t_2,t_1]$ we have $$\Phi(t)=\sum_{k=0}^{k_0-1}\frac{\Phi^{(k)}(t_2)}{k!}(t-t_2)^k+\frac{\Phi^{(k_0)}(\xi(t))}{k_0!}(t-t_2)^{k_0},$$ where $\xi(t) \in (t_2,t)$.

    Since $\Phi^{(k)}(t_2)=L^k_f h(\bar x(t_2))=L^k h(\bar x(t_2))=0$ for each $k \in \{0,\dots,k_0-1\}$ (because $\bar x(t_2) \in S_2$) and $\Phi^{(k_0)}(\xi(t))=L^{k_0}_f h(\bar x(\xi(t)))<0$ for each $t \in (t_2,t_1]$ (because of $(\ref{inequalityChattering})$), we have
    \begin{equation}\label{evenness}
        \Phi(t) = \frac{\Phi^{(k_0)}(\xi(t))}{k_0!}(t-t_2)^{k_0} < 0
    \end{equation}
    for each $t \in (t_2,t_1]$. Thus, $\bar x(t) \in \Omega_-$ for each $t \in (t_2,t_1]$, which contradicts the fact that $\bar x(t) \in \widetilde\Omega_+$ for each $t \in (t_2,t_1]$.

    The case where $k_0$ is even and $t_0 \in (0,\bar T]$ is such that $\bar x(t_0) \in S_2$ can be proved in the same way. The evenness of $k_0$ is used only in $(\ref{evenness})$ to ensure that $(t-t_2)^{k_0}>0$ when $t<t_2$.
\end{proof}

In particular, we proved that a trajectory of $F$ cannot exit $S_2$ directly through $\widetilde\Omega_+$ or through $\widetilde\Omega_-$. Moreover, if $k_0$ is even, it cannot reach $S_2$ directly through either $\widetilde\Omega_+$ or $\widetilde\Omega_-$.
The local behavior of a trajectory of $F$ in each of these cases (provided that they occur) is described in the following corollary.

\begin{corollary}\label{chatteringExit}
    Let $\bar x$ be a trajectory of $F$ defined on an interval $[0,\bar T]$, and let $t_0,t_1 \in [0,\bar T]$ with $t_0<t_1$ be such that $\bar x(t_0) \in S_2$ and $\bar x(t) \notin S_2$ for each $t \in (t_0,t_1]$. Then there exists a strictly decreasing sequence $\{\tau_n\}_{n=1}^{\infty} \subset (t_0,t_1]$ of switching times of $\bar x$, converging to $t_0$, such that, for each $n \ge 1$:
    \begin{enumerate}[label=(\alph*),leftmargin=*]
        \item $\bar x(\tau_n) \in P_o^+$ if $n$ is odd, and $\bar x(\tau_n) \in P_o^-$ if $n$ is even;
        \item $\bar x$ lies in $\widetilde\Omega_-$ on $(\tau_{n+1},\tau_n)$ if $n$ is odd, and $\bar x$ lies in $\widetilde\Omega_+$ on $(\tau_{n+1},\tau_n)$ if $n$ is even.
    \end{enumerate}
    When $k_0$ is even, the analogous statement holds in the reverse time direction: if $t_1<t_0$, $\bar x(t_0) \in S_2$, and $\bar x(t) \notin S_2$ for each $t \in[t_1,t_0)$, then there exists a strictly increasing sequence $\{\tau_n\}_{n=1}^{\infty} \subset [t_1,t_0)$ of switching times of $\bar x$, converging to $t_0$, such that (a) and (b) hold, with $(\tau_{n+1},\tau_n)$ replaced with $(\tau_n,\tau_{n+1})$ for each $n \ge 1$.
\end{corollary}
\begin{proof}
    We prove the first part of the corollary. The case where $k_0$ is even can be proved in the same way.
    
    Since $\bar x(t_0) \in S_2$, we have $L^{k_0}_f h(\bar x(t_0)) < 0 < L^{k_0}_g h(\bar x(t_0))$. By continuity, after possibly shrinking $t_1$, we assume that
    \begin{equation}\label{inequalityChattering1}
        L^{k_0}_f h(\bar x(t)) < 0 < L^{k_0}_g h(\bar x(t))
    \end{equation}
    for each $t \in [t_0,t_1]$.
    
    By Proposition \ref{chattering}, there exists a switching time $\tau_1 \in (t_0,t_1]$ of $\bar x$.
    We assume that $\bar x(\tau_1) \in P^+_o$ (if instead $\bar x(\tau_1)\in P_o^-$, removing $\tau_1$ from the sequence constructed below and renumbering the remaining terms gives a sequence satisfying the conclusion of the corollary).
    According to Proposition \ref{onOmega0}(c), there exists $\delta_1>0$ such that $\bar x(t) \in \widetilde\Omega_-$ for each $t \in (\tau_1-\delta_1,\tau_1)$.
    Let $$\tau_2 := \inf\{t \ge t_0 : \bar x(s) \in \widetilde\Omega_- \text{ for each } s \in (t,\tau_1)\}.$$
    Then $\bar x(t) \in \widetilde\Omega_-$ for each $t \in (\tau_2,\tau_1)$, and $\bar x(\tau_2) \in P_o^-$. Indeed, $\bar x(\tau_2) \notin \widetilde\Omega_+ \cup \widetilde\Omega_-$; otherwise, this would contradict the definition of $\tau_2$. Also, $\bar x(\tau_2) \notin S_2$, since this would contradict Proposition \ref{chattering}, $\bar x(\tau_2) \notin S_1$ by inequality $(\ref{inequalityChattering1})$, and $\bar x(\tau_2) \notin P_o^+$ because of Proposition \ref{onOmega0}(c). Since the family $\{\widetilde\Omega_+,\widetilde\Omega_-,P_o^+,P_o^-,S_1,S_2\}$ forms a partition of $\R^n$, it follows that $\bar x(\tau_2) \in P_o^-$.

    Assume that there exists a positive integer $n_0 \ge 2$ such that $\tau_n$ is defined for each $n \le n_0$.
    We next define $\tau_{n_0+1}$. We first consider the case where $\bar x(\tau_{n_0})\in P_o^+$. According to Proposition \ref{onOmega0}(c), there exists $\delta_{n_0}>0$ such that $\bar x(t)\in\widetilde\Omega_-$ for each $t\in(\tau_{n_0}-\delta_{n_0},\tau_{n_0})$.
    Let $$\tau_{n_0+1} := \inf\{t \ge t_0 : \bar x(s) \in \widetilde\Omega_- \text{ for each } s \in (t,\tau_{n_0})\}.$$  
    Then $\bar x(t) \in \widetilde\Omega_-$ for each $t \in (\tau_{n_0+1},\tau_{n_0})$. As for $\tau_2$ above, it follows that $\bar x(\tau_{n_0+1}) \in P_o^-$.
    The definition of $\tau_{n_0+1}$ in the case where $\bar x(\tau_{n_0})\in P_o^-$ is analogous: we interchange $P_o^+$ and $\widetilde\Omega_-$ with $P_o^-$ and $\widetilde\Omega_+$, and we use case (d) instead of case (c) of Proposition \ref{onOmega0}.

   In this way, we inductively construct a sequence $\{\tau_n\}_{n=1}^{\infty} \subset (t_0,t_1]$ of switching times of $\bar x$. By construction, this sequence is strictly decreasing and satisfies (a) and (b) of the corollary. Since it is also bounded from below by $t_0$, it converges. We set $t_*:=\lim_{n \to \infty} \tau_n$. Then we have $\bar x(t_{*}) \in S_0$ (see Proposition \ref{accumulation}). Since inequality $(\ref{inequalityChattering1})$ holds at $t=t_*$, it follows that $\bar x(t_{*}) \in S_2$. If $t_{*} > t_0$, then $\bar x(t_{*}) \notin S_2$, which is a contradiction. Thus, $t_{*}=t_0$ and the proof is completed.
\end{proof}

Whether trajectories of $F$ with the behavior described in Proposition \ref{chattering} and Corollary \ref{chatteringExit} actually exist remains an open question. As we will show below, the answer depends on the value of $k_0$.

\subsection{The case $k_0=2$}

In this case, $\widetilde\Omega_+=\Omega_+$, $\widetilde\Omega_-=\Omega_-$, the switching sets have the following simpler form:
$$P^+_o = \{x \in \R^n : h(x) = 0 \text{ and } L^1 h(x)>0\},$$
$$P^-_o = \{x \in \R^n : h(x) = 0 \text{ and } L^1 h(x)<0\},$$
and $S_2$ takes the following form:
$$S_2= \left\{x \in \R^n :\, \begin{array}{l}
    h(x)=L^1 h(x)=0,\\
    L^2_f h(x)<0<L^2_g h(x)
\end{array}\right\}.$$

We prove that, for any point of $S_2$, there does not exist a trajectory of $F$ starting from that point. More precisely, we show that consecutive switching times of a trajectory of $F$ cannot accumulate in finite time. In view of Corollary \ref{chatteringExit}, this rules out the existence of trajectories of $F$ intersecting $S_2$. The proof uses ideas from Proposition 15 and Lemma 16 of \cite{BoarottoSigalotti}, adapted to the present setting.

\begin{theorem}\label{order2}
    Let $k_0=2$. If $\bar x$ is a trajectory of $F$ defined on an interval $[0,\bar T]$, then $\bar x(t) \notin S_2$ for each $t \in [0,\bar T]$. In particular, for every $x_0 \in S_2$, there does not exist a trajectory of $F$ starting from $x_0$.
\end{theorem}
\begin{proof}
    Assume the contrary: there exists $t_0 \in [0,\bar T]$ such that $\bar x(t_0) \in S_2$. We first consider the case where $t_0 < \bar T$. Since $L^2_f h(\bar x(t_0)) < 0 <L^2_g h(\bar x(t_0))$, there exists a compact neighborhood $V$ of $\bar x(t_0)$ such that
    \begin{equation}\label{neighborhood}
        L^2_f h(x)<0<L^2_g h(x)
    \end{equation}
    for each $x \in V$. Thus, there exists $\varepsilon>0$ such that $\bar x(t) \in V$ for each $t \in [t_0,t_0+\varepsilon]$.

    As in the proof of Proposition \ref{chattering}, we can show that the trajectory $\bar x$ cannot lie in $S_2$ on $[t_0,t_0+\varepsilon]$. Thus, there exists $t_1 \in (t_0,t_0+\varepsilon]$ such that $\bar x(t_1) \notin S_2$, and hence $\bar x(t_1) \in V \setminus S_2$. As $V \setminus S_2$ is relatively open in $V$, there exists $\delta > 0$ such that $\bar x(t) \in V \setminus S_2$ for each $t \in (t_1-\delta,t_1]$.
    Let $$t_2 := \inf\{t \in [t_0,t_1): \bar x(\tau) \in V \setminus S_2 \text{ for each } \tau \in (t,t_1]\}.$$
    Then $\bar x(t) \in V \setminus S_2$ for each $t \in (t_2,t_1]$ and $\bar x(t_2) \in S_2$; otherwise, $\bar x(t_2) \in V \setminus S_2$, which would contradict the definition of $t_2$.

    By Corollary \ref{chatteringExit}, there exists a strictly decreasing sequence $\{\tau_n\}_{n=1}^{\infty} \subset (t_2,t_1]$ converging to $t_2$ such that:
    \begin{enumerate}[label=(\roman*)]
        \item $\bar x(\tau_n) \in P_o^+$ if $n$ is odd, and $\bar x(\tau_n) \in P_o^-$ if $n$ is even;
        \item $\bar x$ lies in $\Omega_-$ on $(\tau_{n+1},\tau_n)$ if $n$ is odd, and $\bar x$ lies in $\Omega_+$ on $(\tau_{n+1},\tau_n)$ if $n$ is even.
    \end{enumerate}
    
    We set $\Phi(t):=h(\bar x(t))$ for each $t \in [t_2,t_1]$. Let $n \ge 1$ be arbitrary. According to (ii) above, $\bar x$ is a trajectory of $g$ on $[\tau_{n+1},\tau_n]$ if $n$ is odd, and $\bar x$ is a trajectory of $f$ on $[\tau_{n+1},\tau_n]$ if $n$ is even (see Remark \ref{one-sided}). Then Proposition \ref{derivativesY} implies that $\Phi$ is smooth on $[\tau_{n+1},\tau_n]$.
    We use the Taylor expansion of $\Phi$ on the left-hand side of $\tau_n$ with the Lagrange form of the remainder: $$\Phi(\tau_{n+1})=\Phi(\tau_n)+\Phi'(\tau_n)(\tau_{n+1}-\tau_n)+\frac{\Phi''(\xi_n)}{2}(\tau_{n+1}-\tau_n)^2,$$ where $\xi_n \in (\tau_{n+1},\tau_n)$.

    Since $\Phi(\tau_n)=\Phi(\tau_{n+1})=0$ (because $\tau_n$ and $\tau_{n+1}$ are switching times), we have 
    \begin{equation}\label{estimate1}
        \tau_{n+1}-\tau_n=\frac{-2\Phi'(\tau_n)}{\Phi''(\xi_n)}.
    \end{equation}

    From the Mean value theorem applied to $\Phi'$, we have 
    \begin{equation}\label{estimate2}
        \Phi'(\tau_{n+1})=\Phi'(\tau_n)+\Phi''(\zeta_n)(\tau_{n+1}-\tau_n),
    \end{equation} where $\zeta_n \in (\tau_{n+1},\tau_n)$. Then, replacing $(\ref{estimate1})$ in $(\ref{estimate2})$, we have
    \begin{equation}\label{estimate3}
        \Phi'(\tau_{n+1})=\Phi'(\tau_n)\left(1-2\frac{\Phi''(\zeta_n)}{\Phi''(\xi_n)}\right).
    \end{equation}

    Since $\bar x(\tau_n) \in P_o^+$ and $\bar x(\tau_{n+1}) \in P_o^-$, or vice versa, depending on whether $n$ is odd or even, $\Phi'(\tau_n)$ and $\Phi'(\tau_{n+1})$ are nonzero and have different signs. Then $(\ref{estimate3})$ can be rewritten as
    \begin{equation}\label{estimate4}
        \frac{|\Phi'(\tau_{n+1})|}{|\Phi'(\tau_n)|}-1=2\frac{\Phi''(\zeta_n)}{\Phi''(\xi_n)}-2.
    \end{equation}

    Since $n \ge 1$ was arbitrary, $(\ref{estimate4})$ holds for every $n \ge 1$. We will now apply Lemma \ref{log} to the sequence $a_n:=|\Phi'(\tau_n)|$, $n \ge 1$, to obtain a contradiction and complete the proof. For this purpose, by $(\ref{estimate4})$, we need to show that the series $\sum_{n=1}^{\infty} |p_n|$ converges, where $$p_n:=2\frac{\Phi''(\zeta_n)}{\Phi''(\xi_n)}-2 \text{ for each } n \ge 1.$$
    Since the terms of this series are nonnegative, it suffices to show that the partial sums $\sum_{n=1}^{M} |p_n|$, $M \ge 1$, are bounded.
    
    For each $n \ge 1$, the Mean Value Theorem, applied to $\Phi''$ on $[\xi_n,\zeta_n]$ (or on $[\zeta_n,\xi_n]$), implies that there exists $\theta_n \in (\tau_{n+1},\tau_n)$ such that
    \begin{equation}\label{estimate5}
        p_n=2\frac{\Phi''(\zeta_n)-\Phi''(\xi_n)}{\Phi''(\xi_n)}=2\frac{\Phi'''(\theta_n)(\zeta_n-\xi_n)}{\Phi''(\xi_n)}.
    \end{equation}
    By Proposition \ref{derivativesY}, we have $\Phi''(\xi_n)=L^2_{Y_n} h(\bar x(\xi_n))$ and $\Phi'''(\theta_n)=L^3_{Y_n} h(\bar x(\theta_n))$, where $Y_n=g$ if $n$ is odd, and $Y_n=f$ if $n$ is even.

    Since the functions $x \mapsto |L^2_f h(x)|$ and $x \mapsto |L^2_g h(x)|$ are continuous and positive on the compact set $V$ (see $(\ref{neighborhood})$), each of them attains a positive minimum on $V$. Since $\bar x(\xi_n) \in V$ for each $n \ge 1$, it follows that the sequence $\{|\Phi''(\xi_n)|\}_{n=1}^{\infty}$ is bounded from below by a positive number.
    
    Similarly, the functions $x \mapsto L^3_f h(x)$ and $x \mapsto L^3_g h(x)$ are continuous on $\R^n$ (and thus on $V$), so they are bounded on the compact set $V$. Since $\bar x(\theta_n) \in V$ for each $n \ge 1$, it follows that the sequence $\{|\Phi'''(\theta_n)|\}_{n=1}^{\infty}$ is bounded from above.
    
    Using $(\ref{estimate5})$ and combining the above facts, since $\xi_n, \zeta_n \in (\tau_{n+1},\tau_n)$ for each $n \ge 1$, it follows that there exists $L>0$, independent of $n$, such that $$|p_n| \le L|\xi_n-\zeta_n| \le L(\tau_n-\tau_{n+1})$$ for each $n \ge 1$. Consequently, for each $M \ge 1$, $$\sum_{n=1}^M |p_n|\le L \sum_{n=1}^M (\tau_n-\tau_{n+1})\le L(\tau_1-\tau_{M+1}) \le L \bar T.$$ Thus, the partial sums of $\sum_{n=1}^{\infty} |p_n|$ are bounded, and hence $\sum_{n=1}^{\infty} |p_n|$ converges. Therefore, Lemma \ref{log} implies that $a_n=|\Phi'(\tau_n)|$ converges to some positive number, which contradicts the fact that $\lim_{n \to \infty} |\Phi'(\tau_n)| = |\Phi'(t_2)|=0$. The last limit is satisfied because $\Phi$ is of class $C^1$ on $[t_2,t_1]$ (see Proposition \ref{derivatives}), $\lim_{n \to \infty} \tau_n=t_2$, and $\Phi'(t_2)=L^1 h(\bar x(t_2))=0$ (since $\bar x(t_2) \in S_2$). This completes the proof for the case where $t_0 < \bar T$.

    The case where $t_0=\bar T$ is proved similarly, by considering an interval $[t_0-\varepsilon,t_0]$ (instead of $[t_0,t_0+\varepsilon]$) on which $(\ref{neighborhood})$ holds and using the part of Corollary \ref{chatteringExit} for even $k_0$, which is possible because $k_0=2$.
\end{proof}

\begin{lemma}\label{log}
    Let $\{a_n\}_{n=1}^{\infty}\subset(0,+\infty)$. If the series $\sum_{n=1}^{\infty}\left|\frac{a_{n+1}}{a_n}-1\right|$ converges, then $\{a_n\}_{n=1}^{\infty}$ converges to some $a>0$.
\end{lemma}

\begin{proof}
    For each $m\ge1$, since $a_{m+1}=a_1\prod_{n=1}^m\frac{a_{n+1}}{a_n}$, we have
    \begin{equation}\label{telescope}
        \ln a_{m+1}=\ln a_1+\sum_{n=1}^m\ln\frac{a_{n+1}}{a_n}.
    \end{equation}

    Since the series $\sum_{n=1}^{\infty}\left|\frac{a_{n+1}}{a_n}-1\right|$ is convergent, we have $\lim_{n\to\infty}\frac{a_{n+1}}{a_n}=1$. Moreover, since $\lim_{y\to0}\frac{\ln(1+y)}{y}=1$, there exists $N \ge 1$ such that, for every $n \ge N$, $$\left|\ln\frac{a_{n+1}}{a_n}\right|=\left|\ln\left(1+\frac{a_{n+1}}{a_n}-1\right)\right|
    \le 2\left|\frac{a_{n+1}}{a_n}-1\right|.$$
    Therefore, $\sum_{n=N}^{\infty}\left|\ln\frac{a_{n+1}}{a_n}\right|$ is convergent, and hence $\sum_{n=1}^{\infty}\left|\ln\frac{a_{n+1}}{a_n}\right|$ is convergent.
    Thus, the series $\sum_{n=1}^{\infty}\ln\frac{a_{n+1}}{a_n}$ is absolutely convergent.

    Then, from $(\ref{telescope})$, it follows that $\lim_{m\to\infty}\ln a_{m+1}=L$ for some $L \in \R$. Consequently, $\lim_{m\to\infty}a_m=\lim_{m\to\infty}e^{\ln a_m}=e^L>0$.
\end{proof}

\subsection{The case $k_0 \ge 3$}

In this case, we prove that, for each $x_0 \in S_2$, there exists a trajectory of $F$ starting from $x_0$. Given $x_0 \in S_2$, we obtain such a trajectory as a uniform limit of trajectories of $F$ lying outside $S_2$ whose starting points converge to $x_0$ and whose switching points ``move away'' from $S_2$ so that the limit trajectory (which is a trajectory of $\conv F$) stays outside $S_2$ except at the starting time. Then Proposition \ref{coincidence} implies that the limit trajectory is also a trajectory of $F$.

To construct such approximating trajectories, we use the following subsets of the switching sets $P_o^+$ and $P_o^-$, respectively:
$$S_+:=\{x \in \Omega_0 : L^k h(x) > 0 \text{ for each } k \in \{1,\dots,k_0-1\}\}$$
and
$$S_-:=\{x \in \Omega_0 : L^k h(x) < 0 \text{ for each } k \in \{1,\dots,k_0-1\}\}.$$

We first show that each point of $S_2$ can be approximated by points of $S_+ \cup S_-$.

\begin{proposition}\label{closure}
    Let $k_0 \ge 3$. Then $S_2 \subset \overline{S_+} \cap \overline{S_-}$.
\end{proposition}
\begin{proof}
    We first prove that $S_2 \subset \overline{S_+}$. By \ref{A0}, $Dh(x) \ne \mathbf{0}$ for each $x \in \Omega_0$, so there exists an open set $U \supseteq \Omega_0$ such that $Dh(x) \ne \mathbf{0}$ for each $x \in U$.
    We define the vector field $v: U \to \R^n$ by
    \begin{equation}\label{field}
        v(x):=g(x)-\frac{L_g h(x)}{\|Dh(x)\|^2} Dh(x)^T
    \end{equation}
    for each $x \in U$. Since $g$ and $h$ are smooth on $\R^n$ (and thus on $U$), $v$ is smooth on $U$. Moreover, $Dh(x) v(x)=0$ for each $x \in U$, so, for each $x \in \Omega_0$, the vector $v(x)$ is tangent to $\Omega_0$ at $x$.

    Let $x_0 \in S_2$ be arbitrary. Since $v$ is a tangent vector field to the smooth manifold $\Omega_0$, there exist $\varepsilon>0$ and a trajectory $\bar x$ of $v$ defined on $[0,\varepsilon)$ such that $\bar x(0)=x_0$ and $\bar x(t) \in \Omega_0$ for each $t \in [0,\varepsilon)$ (see, e.g., Proposition 9.2 of \cite{Lee}). We will prove that there exists $\varepsilon_0 \in (0,\varepsilon)$ such that $\bar x(t) \in S_+$ for each $t \in (0,\varepsilon_0)$. This will imply that $x_0 \in \overline{S_+}$.
    
    We fix an arbitrary $k \in \{1,\dots,k_0-1\}$. We set $\Phi_k(t):=L^k_g h(\bar x(t))$ for each $t \in [0,\varepsilon)$. The function $\Phi_k$ is smooth on $[0,\varepsilon)$ as a composition of two smooth functions. Moreover, by Proposition \ref{derivativesY}, we have
    \begin{equation}\label{der}
        \Phi_k^{(m)}(t)=L^m_v (L^k_g h)(\bar x(t))
    \end{equation}
    for each $m \in \{0,\dots,k_0-k\}$ and each $t \in [0,\varepsilon)$.

    By Lemma \ref{lemma3} (where $Y(x)=g(x)$ and $Z(x)=-\frac{Dh(x)^T}{\|Dh(x)\|^2}$ for each $x \in U$), for each $m \in \{1,\dots,k_0-k\}$ there exist smooth functions $\mu_{k,m,p}$, $p \in \{1,\dots,m\}$ defined on $U$ such that
    $$L^m_v (L^k_g h)(x)=L^{m+k}_g h(x) + \sum_{p=1}^m \mu_{k,m,p}(x) L^p_g h(x)$$
    for each $x \in U$.
    
    Since $x_0 \in S_2$, $1 \le k \le k_0-1$, and $m \le k_0-k \le k_0-1$, we have $L^p_g h(x_0)=L^p h(x_0)=0$ for each $p \in \{1,\dots,m\}$ and $L^{m+k}_g h(x_0)=0$ whenever $m+k \le k_0-1$. Therefore,
    \begin{equation*}
        \begin{aligned}
            &L^m_v (L^k_g h)(x_0)=0 \text{ for each } m \in \{0,\dots,k_0-k-1\},\\
            &L^{k_0-k}_v (L^k_g h)(x_0)=L^{k_0}_g h(x_0).
        \end{aligned}
    \end{equation*}
    Since $\bar x(0)=x_0 \in S_2$ and $L^{k_0}_g h(x_0)>0$, there exists $\varepsilon_k \in (0,\varepsilon)$ such that $L^{k_0-k}_v (L^k_g h)(\bar x(t)) > 0$ for each $t \in [0,\varepsilon_k)$.
    Hence, by $(\ref{der})$, we have
    \begin{equation}\label{der2}
        \begin{aligned}
            \Phi_k^{(m)}(0)&=0 \text{ for each } m \in \{0,\dots,k_0-k-1\},\\
        \Phi_k^{(k_0-k)}(t)&=L^{k_0-k}_v (L^k_g h)(\bar x(t))>0 \text{ for each } t \in [0,\varepsilon_k).
        \end{aligned}
    \end{equation}

    We use the Taylor expansion of $\Phi_k$ on the right-hand side of zero with the Lagrange form of the remainder: for each $t \in (0,\varepsilon_k)$ we have
    $$\Phi_k(t)=\sum_{m=0}^{k_0-k-1} \frac{\Phi_k^{(m)}(0)}{m!}t^m + \frac{\Phi_k^{(k_0-k)}(\xi(t))}{(k_0-k)!}t^{k_0-k},$$
    where $\xi(t) \in (0,t)$.
    Then, using $(\ref{der2})$, we have $$\Phi_k(t)=\frac{\Phi_k^{(k_0-k)}(\xi(t))}{(k_0-k)!}t^{k_0-k}>0$$ for each $t \in (0,\varepsilon_k)$. Hence, $L^k h(\bar x(t))=L_g^k h(\bar x(t))>0$ for each $t \in (0,\varepsilon_k)$.

    We obtain that for each $k \in \{1,\dots,k_0-1\}$ there exists $\varepsilon_{k} \in (0,\varepsilon)$ such that $L^k h(\bar x(t))>0$ for each $t \in (0,\varepsilon_k)$. We set $\varepsilon_0:=\min_{1 \le k \le k_0-1} \varepsilon_k$. Then $L^k h(\bar x(t))>0$ for each $t \in (0,\varepsilon_0)$. Hence, $\bar x(t) \in S_+$ for each $t \in (0,\varepsilon_0)$. Therefore, $\bar x(0)=x_0 \in \overline{S_+}$. Since $x_0 \in S_2$ was arbitrary, $S_2 \subset \overline{S_+}$.
    The proof that $S_2 \subset \overline{S_-}$ is analogous: replace $g$ and ``$>0$'' with $f$ and ``$<0$'', respectively.
\end{proof}

\begin{lemma}\label{lemma3}
    Let $Y$ and $Z$ be smooth vector fields defined on an open set $U \subseteq \R^n$. Let the vector field $v: U \to \R^n$ be defined by
    \begin{equation}\label{field2}
        v(x):=Y(x)+L_Y h(x) Z(x)
    \end{equation}
    for each $x \in U$. Then, for each $i \ge 1$ and each $j \ge 1$, there exist smooth functions $\mu_{i,j,p}: U \to \R$, $p \in \{1,\dots,j\}$, such that
    \begin{equation}\label{identity}
        L^j_v (L^i_Y h)(x)=L^{i+j}_Y h(x) + \sum_{p=1}^j \mu_{i,j,p}(x) L^p_Y h(x)
    \end{equation}
    for each $x \in U$.
\end{lemma}
\begin{proof}
    We first note that, for all smooth functions $\varphi, \psi : U \to \R$, by the definition of a Lie derivative and the rules of differentiation, we have
    \begin{equation}\label{properties1}
        \begin{aligned}
            &L_v(\varphi+\psi)(x)=L_v \varphi(x)+L_v \psi(x),\\
            &L_v(\varphi \psi)(x)=\psi(x) L_v \varphi(x)+ \varphi(x) L_v \psi(x)
        \end{aligned}
    \end{equation}
    for each $x \in U$.
    Also, by $(\ref{field2})$, we have
    \begin{equation}\label{properties2}
        L_v \varphi(x)=D \varphi(x) (Y+L_Y h \, Z)(x)=L_Y \varphi(x) + L_Y h(x) L_Z \varphi(x)
    \end{equation}
    for each $x \in U$.

    We now fix an arbitrary integer $i \ge 1$. We proceed by induction on $j \ge 1$. For $j=1$, using $(\ref{properties2})$, we have $$L_v(L^i_Y h)(x)=L^{i+1}_Y h(x) + L_Y h(x) L_Z (L^i_Y h)(x)$$ for each $x \in U$. We set $\mu_{i,1,1}(x):=L_Z (L^i_Y h)(x)$ for each $x \in U$. The function $\mu_{i,1,1}$ is smooth on $U$ because $Y$, $Z$ and $h$ are smooth on $U$. Thus, $(\ref{identity})$ holds for $j=1$.

    We now assume that $(\ref{identity})$ holds for some $j \ge 1$. Then (we omit the argument $x$ for clarity of notation):
    \begin{align*}
        L^{j+1}_v (L^i_Y h)&=L_v (L^j_v (L^i_Y h))\\
        &\overset{(\ref{properties1})}{=}L_v (L^{i+j}_Y h) + \sum_{p=1}^j L_v (\mu_{i,j,p} \, L^p_Y h)\\
        &\overset{(\ref{properties1})}{=}L_v (L^{i+j}_Y h) + \sum_{p=1}^j \left(L_v \mu_{i,j,p} \, L^p_Y h + \mu_{i,j,p} \, L_v(L^p_Y h)\right)\\
        &\overset{(\ref{properties2})}{=}L^{i+j+1}_Y h + \nu_{i,j} \, L_Y h + \sum_{p=1}^j L_v \mu_{i,j,p} \, L^p_Y h + \sum_{p=1}^j \mu_{i,j,p} \, L^{p+1}_Y h,
    \end{align*}
    where $\nu_{i,j}:=L_Z (L^{i+j}_Y h)+\sum_{p=1}^j \mu_{i,j,p} \, L_Z (L^p_Y h)$.
    Thus, we obtain $(\ref{identity})$ for $j+1$ with
    \begin{align*}
        \mu_{i,j+1,1}&:=\nu_{i,j}+L_v \mu_{i,j,1},\\
        \mu_{i,j+1,p}&:=\mu_{i,j,p-1}+L_v \mu_{i,j,p},\qquad p \in \{2,\dots,j\},\\
        \mu_{i,j+1,j+1}&:=\mu_{i,j,j}.
    \end{align*}
    Since $Y$, $Z$, $h$, and $\mu_{i,j,p}$, $p \in \{1,\dots,j\}$, are smooth on $U$, the functions $\mu_{i,j+1,p}$, $p \in \{1,\dots,j+1\}$, are smooth on $U$. This completes the induction and the proof.
\end{proof}

We next show that each $x_0 \in S_2$ has a neighborhood $V_{x_0}$ such that the set $V_{x_0} \cap (S_+ \cup S_-)$ is ``forward invariant with respect to switchings'': for each $y \in V_{x_0} \cap (S_+ \cup S_-)$ there exist $\delta>0$ and a unique trajectory $\bar x$ of $F$ on $[0,\delta)$ with $\bar x(0)=y$ that lies in $V_{x_0}$ on $[0,\delta)$ and whose switching times in $(0,\delta)$ all belong to $S_+ \cup S_-$. Moreover, the function $x \mapsto |L^{k_0-1} h(x)|$, evaluated along each such trajectory, increases on the set of successive switching times at a rate that is bounded below by a positive constant, which does not depend on the trajectory in question.

\begin{proposition}\label{invariantSwitchings}
    Let $k_0 \ge 3$ and $x_0 \in S_2$. Then there exist $K>0$ and an open bounded neighborhood $V_{x_0}$ of $x_0$ such that for each $t_1 \ge 0$ and each $x_1 \in V_{x_0} \cap S_\pm$ there exist $t_2 > t_1$ and a unique trajectory $\bar x$ of $F$ on $[t_1,t_2]$ such that $\bar x(t_1)=x_1$, $\bar x(t) \in V_{x_0} \cap \widetilde\Omega_\pm$ for each $t \in (t_1,t_2)$, and either $\bar x(t_2) \in \overline{V_{x_0}} \setminus V_{x_0}$ or $\bar x(t_2) \in V_{x_0} \cap S_\mp$.
    Moreover, if $\bar x(t_2) \in S_\mp$, then
    \begin{equation}\label{increase}
        |L^{k_0-1}h(\bar x(t_2))| \ge |L^{k_0-1}h(\bar x(t_1))|+K(t_2-t_1).
    \end{equation}
\end{proposition}
\begin{proof}
    Since $x_0 \in S_2$, we have $L^{k_0}_f h(x_0) < 0 < L^{k_0}_g h(x_0)$. Thus, there exists an open bounded neighborhood $V_{x_0}$ of $x_0$ such that for each $x' \in V_{x_0}$ and $x'' \in V_{x_0}$ the following inequalities are satisfied:
    \begin{equation}\label{inequality2}
        \frac{L^{k_0}_f h(x')}{L^{k_0}_f h(x'')} > 1-\frac{1}{2k_0} \quad\text{ and }\quad \frac{L^{k_0}_g h(x')}{L^{k_0}_g h(x'')} > 1-\frac{1}{2k_0}.
    \end{equation}
    In particular, $L^{k_0}_f h(x) < 0 < L^{k_0}_g h(x)$ for each $x \in V_{x_0}$.

    \smallskip\noindent\textit{I. Construction of the trajectory $\bar x$ in the case where $x_1 \in V_{x_0} \cap S_+$.}

    Let $t_1 \ge 0$ and $x_1 \in S_+ \cap V_{x_0}$ be arbitrary. We will construct inductively the trajectory $\bar x$ of $F$ with $\bar x(t_1)=x_1$ whose existence is asserted in the proposition.
    
    For each $x \in V_{x_0} \cap (S_+ \cup \widetilde\Omega_+)$ and each $\tau \ge 0$, since $V_{x_0}$ is open, by Proposition \ref{onOmega0}(a) and (c), there exist $\hat{t}_x > 0$ and a unique trajectory $z_{x,\tau}$ of $F$ on $[\tau,\tau+\hat{t}_x)$ such that $z_{x,\tau}(\tau)=x$ and $z_{x,\tau}(t) \in V_{x_0} \cap \widetilde\Omega_+$ for each $t \in (\tau,\tau+\hat{t}_x)$.

    In particular, for the point $x_1$ there exist $\hat t_{x_1}>0$ and a trajectory $z_{x_1,t_1}$ of $F$ on $[t_1,t_1+\hat t_{x_1})$. We set $\bar x(t):=z_{x_1,t_1}(t)$ for each $t \in [t_1,t_1+\hat t_{x_1})$. Then $\bar x(t_1)=x_1$ and $\bar x(t) \in V_{x_0} \cap \widetilde\Omega_+$ for each $t \in (t_1,t_1+\hat t_{x_1})$.
    
    We assume that $\bar x$ is defined on some interval $[t_1,\hat t)$ such that $\bar x(t) \in V_{x_0} \cap \widetilde\Omega_+$ for each $t \in (t_1,\hat t)$. Then $\bar x$ is a trajectory of $f$ on $(t_1,\hat t)$ (see Remark \ref{one-sided}). Since $V_{x_0}$ is bounded and $f$ is continuous on $\R^n$, there exists $c>0$ such that $\|f(x)\| \le c$ for each $x \in V_{x_0}$. Hence, for each increasing sequence $\{\tau_n\}_{n=1}^{\infty} \subset (t_1,\hat t)$ such that $\lim_{n \to \infty} \tau_n = \hat t$, we have
    $$\left\|\bar x(\tau_n)-\bar x(\tau_{n-1})\right\| \le \int_{\tau_{n-1}}^{\tau_n} \left\|\dot{\bar x}(s)\right\|\mathrm{d}s \le c(\tau_n-\tau_{n-1}).$$
    This means that the sequence $\left\{\bar x(\tau_n)\right\}_{n=1}^{\infty}$ is a Cauchy sequence. Hence, there exists $\bar x(\hat{t}) := \lim_{t \to \hat{t}} \bar x(t)$. If $\bar x(\hat t) \in V_{x_0} \cap \widetilde\Omega_+$, we extend $\bar x$ on the interval $(\hat t,\hat t+\hat t_{\bar x(\hat t)})$ by setting $\bar x(t):=z_{{\bar x(\hat t),\hat t}}(t)$ for each $t \in (\hat t, \hat t+\hat t_{\bar x(\hat t)})$. Then $\bar x(t) \in V_{x_0} \cap \widetilde\Omega_+$ for each $t \in (\hat t, \hat t+\hat t_{\bar x(\hat t)})$. Otherwise, if $\bar x(\hat t) \notin V_{x_0} \cap \widetilde\Omega_+$, we terminate the inductive process.

    Let $[t_1,t_2)$ be the maximal interval on which $\bar x$ can be defined in this way. Then $\bar x(t) \in \widetilde\Omega_+$ on $(t_1,t_2)$, and thus $\bar x$ is a trajectory of $f$ on $(t_1,t_2)$ (see Remark \ref{one-sided}). Suppose that $t_2=+\infty$. We set $\Phi(t):=h(\bar x(t))$ for each $t \ge t_1$. Proposition \ref{derivativesY} implies, in particular, that $\Phi$ is of class $C^{k_0}$ on $[t_1,+\infty)$ and $\Phi^{(k)}(t)=L^k_f h(\bar x(t))$ for each $k \in \{1,\dots,k_0\}$ and each $t \ge t_1$. Since $L^{k_0}_f h(x) < \left(1-\frac{1}{2k_0}\right) L^{k_0}_f h(x_0)$ for each $x \in V_{x_0}$ (see $(\ref{inequality2})$), and since $\bar x$ lies in $V_{x_0}$ on $[t_1,+\infty)$, it follows that
    \begin{equation}\label{strict}
        \Phi^{(k_0)}(t)<\left(1-\frac{1}{2k_0}\right) L^{k_0}_f h(x_0)<0
    \end{equation} for each $t \ge t_1$.

    We use the Taylor expansion of $\Phi$ on the right-hand side of $t_1$ with the Lagrange form of the remainder: for each $t > t_1$ we have
    \begin{equation}\label{expansion}
        \Phi(t)=\sum_{k=0}^{k_0-1}\frac{\Phi^{(k)}(t_1)}{k!}(t-t_1)^k+\frac{\Phi^{(k_0)}(\xi(t))}{k_0!}(t-t_1)^{k_0},
    \end{equation}
    where $\xi(t) \in (t_1,t)$.
    Using $(\ref{strict})$, we have
    \begin{align*}
        \Phi(t) &< \sum_{k=0}^{k_0-1}\frac{\Phi^{(k)}(t_1)}{k!}(t-t_1)^k + \frac{1}{k_0!} \left(1-\frac{1}{2k_0}\right) L^{k_0}_f h(x_0) (t-t_1)^{k_0}\\
        &< \frac{1}{2 \,k_0!} \left(1-\frac{1}{2k_0}\right) L^{k_0}_f h(x_0) (t-t_1)^{k_0}
    \end{align*}
    for $t$ sufficiently large.
    Thus, $\lim_{t \to +\infty} \Phi(t)=-\infty$, which is impossible, since $\bar x$ lies in $V_{x_0}$ on $[t_1,+\infty)$ and $h$ is bounded on $V_{x_0}$. Therefore, $t_2 < +\infty$.

    Let $\bar x(t_2):=\lim_{t \to t_2-} \bar x(t)$ (the limit exists as is seen above). Since $\bar x$ lies in $V_{x_0}$ on $[t_1,t_2)$, we have $\bar x(t_2) \in \overline{V_{x_0}}$. If $\bar x(t_2) \in \overline{V_{x_0}} \setminus V_{x_0}$, the construction of the trajectory is completed. If $\bar x(t_2) \in V_{x_0}$, then $\bar x(t_2) \in \Omega_0$. Indeed, if $\bar x(t_2) \in \Omega_-$, then $\bar x(t) \in \Omega_-$ for $t$ close to $t_2$, which contradicts $\bar x(t) \in \widetilde\Omega_+$ on $(t_1,t_2)$. If $\bar x(t_2) \in \Omega_+$, then $\bar x$ could be extended beyond $t_2$, which contradicts the maximality of $[t_1,t_2)$.

    It remains to show that if $\bar x(t_2) \in V_{x_0}$, then $\bar x(t_2) \in S_-$.
    Since $\bar x(t_1) \in S_+ \subseteq \Omega_0$ and $\bar x(t_2) \in \Omega_0$, we have $\Phi(t_1)=\Phi(t_2)=0$. Also, since $\bar x(t_1) \in S_+$ and $\bar x$ lies in $V_{x_0}$ on $[t_1,t_2]$, we have $\Phi^{(k)}(t_1)=L^k_f h(\bar x(t_1))=L^k h(\bar x(t_1))>0$ for each $k \in \{1,\dots,k_0-1\}$ and $\Phi^{(k_0)}(t)<0$ for each $t \in [t_1,t_2]$.

    We will prove that $\Phi^{(k)}(t_2)<0$ for each $k \in \{1,\dots,k_0-1\}$. More precisely, by induction, we will show that, for each $k \in \{1,\dots,k_0-1\}$, $\Phi^{(k)}$ changes its sign exactly once on $(t_1,t_2)$, that is, there exists $s_k \in (t_1,t_2)$ such that $\Phi^{(k)}(t)>0$ for each $t \in [t_1,s_k)$, $\Phi^{(k)}(s_k)=0$, and $\Phi^{(k)}(t)<0$ for each $t \in (s_k,t_2]$.

    Since $\Phi^{(k_0)}(t) < 0$ for each $t \in [t_1,t_2]$, $\Phi^{(k_0-1)}$ is strictly decreasing on $[t_1,t_2]$. Suppose that $\Phi^{(k_0-1)}(t_2) \ge 0$. Since $\Phi^{(k_0-1)}(t_1) > 0$, it follows that $\Phi^{(k_0-1)}(t) > 0$ for each $t \in [t_1,t_2)$. Using the Taylor expansion of $\Phi$ on the right-hand side of $t_1$ with the Lagrange form of the remainder, we have
    $$\Phi(t_2)=\sum_{k=0}^{k_0-2}\frac{\Phi^{(k)}(t_1)}{k!}(t_2-t_1)^k+\frac{\Phi^{(k_0-1)}(\xi_{k_0-1})}{(k_0-1)!}(t_2-t_1)^{k_0-1},$$ where $\xi_{k_0-1} \in (t_1,t_2)$. Since $\Phi^{(k)}(t_1) \ge 0$ for each $k \in \{0,\dots,k_0-2\}$ and $\Phi^{(k_0-1)}(t) > 0$ for each $t \in (t_1,t_2)$, it follows that $\Phi(t_2)>0$, which contradicts $\Phi(t_2)=0$. Thus, $\Phi^{(k_0-1)}(t_2) < 0$ and there exists $s_{k_0-1} \in (t_1,t_2)$ such that $\Phi^{(k_0-1)}(t)>0$ for each $t \in [t_1,s_{k_0-1})$, $\Phi^{(k_0-1)}(s_{k_0-1})=0$, and $\Phi^{(k_0-1)}(t)<0$ for each $t \in (s_{k_0-1},t_2]$.

    We assume that for some $k \in \{2,\dots,k_0-1\}$ there exists $s_k \in (t_1,t_2)$ such that $\Phi^{(k)}(t)>0$ for each $t \in [t_1,s_k)$, $\Phi^{(k)}(s_k)=0$, and $\Phi^{(k)}(t)<0$ for each $t \in (s_k,t_2]$.
    Then $\Phi^{(k-1)}$ is strictly increasing on $(t_1,s_k)$ and strictly decreasing on $(s_k,t_2)$. Suppose that $\Phi^{(k-1)}(t_2) \ge 0$. Since $\Phi^{(k-1)}(t_1) > 0$, it follows that $\Phi^{(k-1)}(t) > 0$ for each $t \in [t_1,t_2)$. As in the base case, we obtain $\Phi(t_2)>0$, which contradicts $\Phi(t_2)=0$. Thus, $\Phi^{(k-1)}(t_2) < 0$ and there exists $s_{k-1} \in (t_1,t_2)$ such that $\Phi^{(k-1)}(t)>0$ for each $t \in [t_1,s_{k-1})$, $\Phi^{(k-1)}(s_{k-1})=0$, and $\Phi^{(k-1)}(t)<0$ for each $t \in (s_{k-1},t_2]$. This completes the induction. Consequently, for each $k \in \{1,\dots,k_0-1\}$, $\Phi^{(k)}(t_2)<0$, that is, $L^k h(\bar x(t_2))=L^k_f h(\bar x(t_2)) < 0$. Thus, $\bar x(t_2) \in S_-$.

    The constructed trajectory $\bar x$ is unique. Indeed, let $\widetilde x$ be a trajectory of $F$ on $[t_1,t_2]$ such that $\widetilde x(t_1)=x_1$ and $\widetilde x(t) \in V_{x_0} \cap \widetilde\Omega_+$ for each $t \in (t_1,t_2)$, and suppose that $\widetilde x \ne \bar x$. By the local uniqueness in Proposition \ref{onOmega0}(c), there exists $\delta>0$ such that $\widetilde x=\bar x$ on $[t_1,t_1+\delta)$. We set $t_0:=\sup\{t \in [t_1,t_2]: \widetilde x=\bar x \text{ on } [t_1,t)\}$.
    Then $t_1<t_0<t_2$ and $\widetilde x(t_0)=\bar x(t_0)$ by continuity. Since $t_0 \in (t_1,t_2)$, we have $\bar x(t_0) \in V_{x_0} \cap \widetilde\Omega_+$. By the local uniqueness in Proposition \ref{onOmega0}(a), we have $\widetilde x=\bar x$ on $[t_0,t_0+\delta_0)$ for some $\delta_0>0$, which contradicts the definition of $t_0$. Hence, $\widetilde x=\bar x$ on $[t_1,t_2]$.
    
    \smallskip\noindent\textit{II. Proof of the inequality $(\ref{increase})$ in the case where $x_1 \in V_{x_0} \cap S_+$.}

    We assume that $\bar x(t_2) \in S_-$.
    We first show that $$|\Phi^{(k_0-1)}(t_2)| > (k_0-3/2)|\Phi^{(k_0-1)}(t_1)|.$$
    From the Taylor expansion of $\Phi$ on the right-hand side of $t_1$ in $(\ref{expansion})$, for some $t_*:=\xi(t_2) \in (t_1,t_2)$ we have: $$\Phi(t_2)=\sum_{k=0}^{k_0-1}\frac{\Phi^{(k)}(t_1)}{k!}(t_2-t_1)^k+\frac{\Phi^{(k_0)}(t_*)}{k_0!}(t_2-t_1)^{k_0}.$$
    Since $\Phi(t_1)=\Phi(t_2)=0$ and $\Phi^{(k)}(t_1)>0$ for each $k \in \{1,\dots,k_0-2\}$ (because $\bar x(t_1)=x_1 \in S_+$), we have $$\frac{\Phi^{(k_0-1)}(t_1)}{(k_0-1)!}(t_2-t_1)^{k_0-1}+\frac{\Phi^{(k_0)}(t_*)}{k_0!}(t_2-t_1)^{k_0}<0.$$
    Since $\Phi^{(k_0)}(t_*)<0$ (because $\bar x$ lies in $V_{x_0}$ on $(t_1,t_2)$), we have
    \begin{equation}\label{inequality1}
        t_2-t_1 > k_0 \frac{\Phi^{(k_0-1)}(t_1)}{-\Phi^{(k_0)}(t_*)}.
    \end{equation}

    From the Mean value theorem applied to $\Phi^{(k_0-1)}$ on the interval $[t_1,t_2]$, we have
    \begin{equation}\label{estimate6}
        \Phi^{(k_0-1)}(t_2)=\Phi^{(k_0-1)}(t_1)+\Phi^{(k_0)}(t_{**})(t_2-t_1),
    \end{equation}
    for some $t_{**} \in (t_1,t_2)$. Using $(\ref{inequality1})$, since $\Phi^{(k_0)}(t_{**})<0$, we have $$\Phi^{(k_0-1)}(t_2)<\Phi^{(k_0-1)}(t_1) \left(1-k_0 \frac{\Phi^{(k_0)}(t_{**})}{\Phi^{(k_0)}(t_*)}\right).$$
    Since $\Phi^{(k_0-1)}(t_1)>0$, $\Phi^{(k_0-1)}(t_2)<0$, and $(\ref{inequality2})$ holds, we obtain
    \begin{equation}\label{inequality3}
        |\Phi^{(k_0-1)}(t_2)|>(k_0-3/2)|\Phi^{(k_0-1)}(t_1)|.
    \end{equation}

    It remains to find $K>0$, depending only on $x_0$, such that $|\Phi^{(k_0-1)}(t_2)|>|\Phi^{(k_0-1)}(t_1)|+K(t_2-t_1)$.
    From $(\ref{estimate6})$, since $\Phi^{(k_0-1)}(t_1)>0$ and $\Phi^{(k_0-1)}(t_2)<0$, we obtain $$t_2-t_1=\frac{\Phi^{(k_0-1)}(t_2)-\Phi^{(k_0-1)}(t_1)}{\Phi^{(k_0)}(t_{**})}=\frac{|\Phi^{(k_0-1)}(t_1)|+|\Phi^{(k_0-1)}(t_2)|}{-\Phi^{(k_0)}(t_{**})}.$$
    Then we have $$t_2-t_1 \overset{(\ref{inequality3})}{<} \frac{2k_0-1}{2k_0-3} \frac{|\Phi^{(k_0-1)}(t_2)|}{-\Phi^{(k_0)}(t_{**})}\overset{(\ref{inequality2})}{<}\frac{2k_0}{2k_0-3}\frac{|\Phi^{(k_0-1)}(t_2)|}{-L^{k_0}_f h(x_0)}.$$
    Thus,
    \begin{align*}
        \frac{|\Phi^{(k_0-1)}(t_2)|-|\Phi^{(k_0-1)}(t_1)|}{t_2-t_1}&>-\frac{2k_0-3}{2k_0}L^{k_0}_f h(x_0)\left(1-\frac{|\Phi^{(k_0-1)}(t_1)|}{|\Phi^{(k_0-1)}(t_2)|}\right)\\
        &\overset{(\ref{inequality3})}{>}-\frac{2k_0-5}{2k_0}L^{k_0}_f h(x_0).
    \end{align*}
    We set $K_f:=-\frac{2k_0-5}{2k_0}L^{k_0}_f h(x_0)>0$. Then $|\Phi^{(k_0-1)}(t_2)|>|\Phi^{(k_0-1)}(t_1)|+K_f(t_2-t_1)$. Thus, $|L^{k_0-1}(\bar x(t_2))|>|L^{k_0-1}(\bar x(t_1))|+K_f(t_2-t_1)$.

     \smallskip\noindent\textit{III. The case where $x_1 \in V_{x_0} \cap S_-$ and the choice of the constant $K$.}

     The proof is analogous to the case where $x_1\in V_{x_0}\cap S_+$. We replace $f$ and $\widetilde\Omega_+$ with $g$ and $\widetilde\Omega_-$, respectively. We interchange $S_+$ with $S_-$ and $\Omega_+$ with $\Omega_-$. We use Proposition~\ref{onOmega0}(b) and (d) instead of (a) and (c). The uniqueness of the trajectory is proved in the same way. The same computation as in part II gives the inequality \eqref{increase} with the constant $K_g:=\frac{2k_0-5}{2k_0}L^{k_0}_gh(x_0)>0$.
     Hence, \eqref{increase} holds in both cases with $K:=\min\{K_f,K_g\}$, which depends only on $x_0$.
\end{proof}

For each $x_0 \in S_2$, using Proposition \ref{closure} and Proposition \ref{invariantSwitchings}, we can now construct a uniformly convergent sequence of trajectories of $F$ whose starting points converge to $x_0$. The uniform growth estimate $(\ref{increase})$ in Proposition \ref{invariantSwitchings} ensures that the limit trajectory stays outside $S_2$ after the starting time. Then Proposition \ref{coincidence} implies that it is a trajectory of $F$.

\begin{theorem}
    Let $k_0 \ge 3$. Then for each $x_0 \in S_2$, there exists a trajectory of $F$ starting from $x_0$.
\end{theorem}
\begin{proof}
    We fix an arbitrary point $x_0 \in S_2$. We denote by $V_{x_0}$ the open bounded neighborhood of $x_0$ and by $K$ the constant obtained in Proposition \ref{invariantSwitchings}.


    \smallskip\noindent\textit{I. Construction of approximating trajectories $y_m$ of $F$, $m \ge 1$, with finitely many switching points.}
    
    For each $x \in V_{x_0} \cap S_\pm$ and each $\tau \ge 0$, by Proposition \ref{invariantSwitchings}, there exist $\hat t_x >0$ and a unique trajectory $z_{x,\tau}$ of $F$ on $[\tau,\tau+\hat t_x]$ such that $z_{x,\tau}(\tau)=x$, $z_{x,\tau}(t) \in V_{x_0} \cap \widetilde\Omega_\pm$ for each $t \in (\tau,\tau+\hat t_x)$, and either $z_{x,\tau}(\tau+\hat t_x) \in \overline{V_{x_0}} \setminus V_{x_0}$ or $z_{x,\tau}(\tau+\hat t_x) \in V_{x_0} \cap S_\mp$. Also, if $z_{x,\tau}(\tau+\hat t_x) \in V_{x_0} \cap S_\mp$, then $$|L^{k_0-1}h(z_{x,\tau}(\tau+\hat t_x))| \ge |L^{k_0-1}h(z_{x,\tau}(\tau))|+K \hat t_x.$$
    
    We fix an arbitrary positive integer $m$. By Proposition \ref{closure}, there exists a point $x_m \in V_{x_0} \cap (S_+ \cup S_-)$ such that $\|x_m-x_0\|<\frac{1}{m}$. 
    We construct a finite sequence of times $0 = s^m_0 < s^m_1 < \dots < s^m_{j(m)} = T_m$ and a trajectory $y_m$ of $F$ on $[0,T_m]$ such that $y_m(0)=x_m$, $y_m(t) \in V_{x_0}$ for each $t \in [0,T_m)$, $y_m(T_m) \in \overline{V_{x_0}} \setminus V_{x_0}$, and \begin{equation}\label{switchInequality}
        |L^{k_0-1} h(y_m(s^m_j))| \ge |L^{k_0-1} h(x_m)|+K s^m_j
    \end{equation}
    for each $j \in \{0,\dots,j(m)-1\}$.
    The times $s^m_1, \dots ,s^m_{j(m)-1}$ will be exactly the switching times of $y_m$ in $[0,T_m]$.

    We set $s^m_0:=0$ and $y_m(0):=x_m$. 
    We assume that, for some $i \ge 0$, the times $s^m_0 < \dots < s^m_i$ and the trajectory $y_m$ on $[0,s^m_i]$ have been constructed so that $y_m(t) \in V_{x_0}$ for each $t \in [0,s^m_i)$ and $(\ref{switchInequality})$ holds for each $j \in \{0,\dots, i-1\}$. We also assume that either $y_m(s^m_i) \in \overline{V_{x_0}} \setminus V_{x_0}$, or $y_m(s^m_i) \in V_{x_0} \cap (S_+ \cup S_-)$ and $(\ref{switchInequality})$ holds also for $j=i$. For $i=0$ this hypothesis holds, since $y_m(0)=x_m \in V_{x_0} \cap (S_+\cup S_-)$ and $(\ref{switchInequality})$ is an equality for $j=0$.
    
    If $y_m(s^m_i) \in \overline{V_{x_0}} \setminus V_{x_0}$, we set $j(m):=i$ and $T_m:=s^m_i$, and we terminate the inductive process. Otherwise, if $y_m(s^m_i) \in V_{x_0} \cap (S_+ \cup S_-)$, we set $s^m_{i+1}:=s^m_i+\hat t_{y_m(s^m_i)}$ and we extend $y_m$ by setting $y_m(t):=z_{y_m(s^m_i),s^m_i}(t)$ for each $t \in (s^m_i,s^m_{i+1}]$. Then $y_m(t) \in V_{x_0}$ for each $t \in [s^m_i,s^m_{i+1})$, and either $y_m(s^m_{i+1}) \in \overline{V_{x_0}} \setminus V_{x_0}$ or $y_m(s^m_{i+1}) \in V_{x_0} \cap (S_+ \cup S_-)$. Also, if $y_m(s^m_{i+1}) \in V_{x_0} \cap (S_+ \cup S_-)$, we have 
    \begin{align*}
        |L^{k_0-1} h(y_m(s^m_{i+1}))|& \ge |L^{k_0-1} h(y_m(s^m_i))| + K(s^m_{i+1}-s^m_i)\\
        &\ge |L^{k_0-1} h(x_m)| + K s^m_i + K(s^m_{i+1}-s^m_i)\\
        &=|L^{k_0-1} h(x_m)|+ K s^m_{i+1},
    \end{align*}
    and thus $(\ref{switchInequality})$ holds for $j=i+1$.

    The above-described inductive process must terminate after finitely many steps. Assume the contrary. Then there exists an infinite strictly increasing sequence $\{s^m_j\}_{j=0}^{\infty}$ such that $y_m$ is defined on $[0,s^m_*)$, where $s^m_*:=\lim_{j \to \infty} s^m_j \in \R \cup \{+\infty\}$. Moreover, $y_m(t) \in V_{x_0}$ for each $t \in [0,s^m_*)$, and $(\ref{switchInequality})$ holds for each $j \ge 0$.
    If $s^m_* < +\infty$, then $y_m(s^m_*):=\lim_{t \to s^m_*-} y_m(t)$ exists (it can be shown in the same way as in part I of the proof of Proposition \ref{invariantSwitchings}). By taking the limit as $j \to \infty$ in $(\ref{switchInequality})$, we obtain $|L^{k_0-1} h(y_m(s^m_*))| \ge |L^{k_0-1} h(x_m)|+Ks^m_*>0$, and thus $y_m(s^m_*) \notin S_0$. This is a contradiction because each accumulation point of switching points ($y_m(s^m_*)$ is such an accumulation point) necessarily belongs to the singular set $S_0$ (see Proposition \ref{accumulation}). Hence, $s^m_* = +\infty$. By taking the limit as $j \to \infty$ in $(\ref{switchInequality})$, we obtain $\lim_{j \to \infty} |L^{k_0-1} h(y_m(s^m_j))|=+\infty$, which is impossible since the function $x \mapsto L^{k_0-1} h(x)$ is bounded on $V_{x_0}$ and $y_m$ lies in $V_{x_0}$ on $[0,s^m_*)$. Therefore, the inductive process terminates after finitely many steps, that is, there exists a positive integer $j(m)$ such that $y_m(s^m_{j(m)}) \in \overline{V_{x_0}} \setminus V_{x_0}$. Then we set $T_m:=s^m_{j(m)}$.

    Let $C>0$ be an upper bound for the values of $F$ on $\overline{V_{x_0}}$ (such exists because $V_{x_0}$ is bounded and $f$ and $g$ are continuous on $\R^n$). Since $y_m(t) \in \overline{V_{x_0}}$ for each $t \in [0,T_m]$, we have $$\left\|y_m(T_m)-x_m\right\| = \left\|\int_{0}^{T_m} \dot{y}_m(\tau)\mathrm{d}\tau\right\| \le \int_{0}^{T_m} \left\|\dot{y}_m(\tau)\right\|\mathrm{d}\tau \le C T_m.$$
    Since $y_m(T_m) \in \overline{V_{x_0}} \setminus V_{x_0}$, $\mathrm{dist}\left(x_m,\mathbb{R}^n \setminus V_{x_0}\right) \le \|y_m(T_m)-x_m\| \le C T_m$. Since $V_{x_0}$ is open and $x_m \in V_{x_0}$, we have
    \begin{equation}\label{time}
        T_m \ge \frac{\mathrm{dist}(x_m,\mathbb{R}^n \setminus V_{x_0})}{C} > 0.
    \end{equation}

    \smallskip\noindent\textit{II. Uniform convergence to a trajectory $y$ of $\conv F$ on an interval $[0,T]$.}

    Repeating the above-described inductive process for each positive integer $m$, we obtain a sequence $\{y_m\}_{m=1}^{\infty}$ of trajectories of $F$ such that, for each $m \ge 1$, $y_m$ is defined on $[0,T_m]$, $y_m(0)=x_m$, $y_m(t) \in V_{x_0}$ for each $t \in [0,T_m)$, $y_m(T_m) \in \overline{V_{x_0}} \setminus V_{x_0}$, and $|L^{k_0-1} h(y_m(s^m_j))| \ge |L^{k_0-1} h(x_m)|+K s^m_j$ for each $j \in \{0,\dots,j(m)-1\}$, where $s^m_1 < \dots < s^m_{j(m)-1}$ are the finitely many switching times of $y_m$ in $[0,T_m]$.

    Since $\|x_m-x_0\|<\frac{1}{m}$ for each integer $m \ge 1$, it follows that $\lim_{m \to \infty} x_m=x_0$. Then there exists a positive integer $m_0$ such that for each $m \ge m_0$ we have $\|x_m-x_0\| \le \frac{1}{2} \mathrm{dist}(x_0,\mathbb{R}^n \setminus V_{x_0})$. Using the lipschitzness of the distance function and the fact that $V_{x_0}$ is open, for each $m \ge m_0$ we have
    \begin{equation*}
        \begin{aligned}
            T_m \overset{(\ref{time})}{\ge} \frac{\mathrm{dist}(x_m,\mathbb{R}^n \setminus V_{x_0})}{C} & \ge \frac{\mathrm{dist}(x_0,\mathbb{R}^n \setminus V_{x_0}) - \|x_m-x_0\|}{C}\\ &\ge T:= \frac{\mathrm{dist}(x_0,\mathbb{R}^n \setminus V_{x_0})}{2C} > 0.
        \end{aligned}
    \end{equation*}

    Since $F$ is bounded on $\overline{V_{x_0}}$, the sequence $\{y_m\}_{m=m_0}^{\infty}$ of functions defined on the interval $[0,T]$ is uniformly bounded and equicontinuous. Applying the Arzel\`{a}-Ascoli theorem, without loss of generality, we assume that this sequence converges uniformly to an absolutely continuous function $y : [0,T] \to \overline{V_{x_0}}$.

    Since $y_m$ is a trajectory of $F$ on $[0,T]$ for each integer $m \ge m_0$, and since $F$ is upper semi-continuous with nonempty compact values on $\overline{V_{x_0}}$, it follows that $y$ is a trajectory of $\conv F$ on $[0,T]$, that is, $\dot{y}(t) \in \conv F(y(t))$ for almost every $t \in [0,T]$ (see, e.g., Corollary 2 on page 77 of \cite{Filippov88}). We also have $y(0)=\lim_{m \to \infty} y_m(0)=\lim_{m \to \infty} x_m=x_0$.

    \smallskip\noindent\textit{III. The limit trajectory $y$ is a trajectory of $F$ on $[0,T]$.}

    We will show that $y(t) \notin S_2$ for each $t \in (0,T)$. Then from Corollary \ref{coincidence} it will follow that $y$ is a trajectory of $F$ on $[0,T]$, which will complete the proof.

    Suppose that there exists $t_0 \in (0,T)$ such that $y(t_0) \in S_2$. For each integer $m \ge m_0$, we set 
    $$\hat s_m:=\min\{s^m_j \in [t_0,T]: j \in \{1,\dots,j(m)-1\}\},$$ if $y_m$ has a switching time in $[t_0,T]$; otherwise, we set $\hat s_m:=T$. The minimum exists in the former case, since $y_m$ has finitely many switching times in $[0,T_m] \supseteq [t_0,T]$.

    Since $\hat s_m \in [t_0,T]$ for each $m \ge m_0$, without loss of generality, we assume that the sequence $\{\hat s_m\}_{m=m_0}^{\infty}$ converges to some $\hat s_0 \in [t_0,T]$. Suppose that $\hat s_0 > t_0$. Then there exist an integer $m_1 \ge m_0$ and $\delta \in (0,\hat s_0-t_0)$ such that $\hat s_m > t_0+\delta$ for each $m \ge m_1$. Thus, $y_m$ has no switching times on the interval $[t_0,t_0+\delta]$ for each $m \ge m_1$. Therefore, by the construction in part I, for each $m \ge m_1$, either $y_m$ lies in $\widetilde\Omega_+$ on $[t_0,t_0+\delta]$ or $y_m$ lies in $\widetilde\Omega_-$ on $[t_0,t_0+\delta]$.

    Without loss of generality, we assume that, for each $m \ge m_1$, $y_m$ lies in $\widetilde\Omega_+$ on $[t_0,t_0+\delta]$. Then, for each $m \ge m_1$, $h(y_m(t)) \ge 0$ for each $t \in [t_0,t_0+\delta]$, and $y_m$ is a trajectory of $f$ on $[t_0,t_0+\delta]$ (see Remark \ref{one-sided}). Since $\{y_m\}_{m=m_1}^{\infty}$ converges uniformly to $y$ on $[t_0,t_0+\delta]$, it follows that $h(y(t)) \ge 0$ for each $t \in [t_0,t_0+\delta]$ and that $y$ is a trajectory of $f$ on $[t_0,t_0+\delta]$. Hence, $y(t) \in \Omega_0 \cup \Omega_+$ for each $t \in [t_0,t_0+\delta]$. Since $f(x) \in F(x)$ for each $x \in \Omega_0 \cup \Omega_+$, it follows that $y$ is a trajectory of $F$ on $[t_0,t_0+\delta]$. Moreover, $y$ has no switching times in $(t_0,t_0+\delta)$, since $y$ stays outside $\Omega_-$ on $[t_0,t_0+\delta]$. As $y(t_0) \in S_2$, this contradicts Proposition \ref{chattering}. Consequently, $\hat s_0=t_0$.
    
    Since $\lim_{m \to \infty} \hat s_m=t_0 < T$, without loss of generality, we assume that $\hat s_m < T$ for each $m \ge m_1$. By the definition of $\hat s_m$, this means that $\hat s_m$ is a switching time of $y_m$ for each $m \ge m_1$. Thus, from the construction of the trajectory $y_m$ in part I, we have
    \begin{equation}\label{inequality4}
        |L^{k_0-1} h(y_m(\hat s_m))| \ge |L^{k_0-1} h(x_m)|+K \hat s_m
    \end{equation}
    for each $m \ge m_1$.
    Since $y_m$, $m \ge m_1$, are Lipschitz on $V_{x_0}$ with a common Lipschitz constant, $\lim_{m \to \infty} \hat s_m=t_0$ and $\lim_{m \to \infty} y_m(t_0)=y(t_0)$, we have $\lim_{m \to \infty} y_m(\hat s_m) = y(t_0)$. Then, by taking the limit in $(\ref{inequality4})$ as $m \to \infty$, and using $\lim_{m \to \infty} x_m=x_0 \in S_2$ and the continuity of $L^{k_0-1} h$, we get $$|L^{k_0-1} h(y(t_0))| \ge K t_0 > 0.$$
    Since $L^{k_0-1} h$ is identically zero on $S_2$, it follows that $y(t_0) \notin S_2$, which contradicts $y(t_0) \in S_2$. Thus, $y(t) \notin S_2$ for each $t \in (0,T)$. Then Corollary \ref{coincidence} implies that $y$ is a trajectory of $F$ on $[0,T]$. This completes the proof.
\end{proof}

\section{An example}

In this section, we show that Theorem \ref{main} can be applied to a class of differential inclusions arising from the following linear-quadratic optimal control problem.

\medskip
\noindent [LQ] \quad Minimize 
\begin{equation*} 
    \int_0^T x(t)^T Q x(t) \mathrm{d}t
\end{equation*}
subject to
\begin{align}
    \label{lineardynamics}\dot x(t) &= Ax(t) + u(t)b,\ \ u(t)\in [-1,1],\\
    \nonumber x(0)&=x_0 \text{ and } x(T)=x_1.
\end{align}

Here, $A \in \R^{n \times n}$, $b \in \R^n$ is nonzero, and $Q \in \R^{n \times n}$ is symmetric, positive semi-definite, and nonzero. The starting point $x_0 \in \R^n$ and the target point $x_1 \in \R^n$ are arbitrarily fixed. The terminal time $T>0$ is free.

The admissible controls are all Lebesgue measurable functions that are defined on an interval $[0,T]$ for some $T>0$ and take values in $[-1,1]$. An admissible trajectory of $(\ref{lineardynamics})$ corresponding to the admissible control $u: [0,T] \to [-1,1]$ is an absolutely continuous function $x: [0,T] \to \R^n$ such that $(\ref{lineardynamics})$ is satisfied for almost every $t \in [0,T]$. The pair $(x,u)$ is called an admissible trajectory-control pair of $(\ref{lineardynamics})$ defined on $[0,T]$.

The Pontryagin maximum principle provides necessary conditions for optimality. Consequently, the admissible trajectory-control pairs satisfying these conditions are the candidates for optimal solutions of [LQ]. For this problem, these conditions take the following form (see, e.g., Theorem 4.2.3 of \cite{SchattlerLedzewicz}).
If $(\bar x, \bar u)$ is an optimal trajectory-control pair of [LQ] defined on $[0,\bar T]$, then there exist $\bar\lambda_0 \in \{0,1\}$ and an absolutely continuous function $\bar\lambda:[0,\bar T]\to \R^n$ such that $(\bar\lambda_0,\bar\lambda(t)) \ne (0,\mathbf{0})$ for each $t \in [0,\bar T]$,
\begin{equation}\label{PMP}
    \begin{aligned}
        \dot{\bar\lambda}(t) = -A^T\bar\lambda(t)+2\bar\lambda_0 Q \bar x(t),\\
        \bar\lambda(t)^T b \bar u(t) = \max\limits_{u \in [-1,1]} \bar\lambda(t)^T b u
    \end{aligned}
\end{equation}
for almost every $t \in [0,\bar T]$, and, since $\bar T$ is free,
\begin{equation}\label{HamiltonianZero}
    -\bar\lambda_0 \bar x(t)^T Q \bar x(t)+\bar\lambda(t)^T A \bar x(t) + |\bar\lambda(t)^T b| = 0
\end{equation}
for every $t \in [0,\bar T]$.

If $(\bar x,\bar u)$ is an admissible trajectory-control pair of $(\ref{lineardynamics})$ defined on $[0,\bar T]$ for which there exist $\bar\lambda_0 \in \{0,1\}$ and an absolutely continuous function $\bar\lambda: [0,\bar T] \to \R^n$ satisfying the conditions of the Pontryagin maximum principle, then $\bar x$ is called an {\it extremal trajectory} and $\bar u$ is called an {\it extremal control}. In such a case, the maximization condition in $(\ref{PMP})$ uniquely determines $\bar u$ on the set $\{t \in [0,\bar T] : \bar\lambda(t)^T b \ne 0\}$ (up to a modification on a set of measure zero) by $\bar u(t) = \mathrm{sgn}\,(\bar\lambda(t)^T b)$, while on $\{t \in [0,\bar T] : \bar\lambda(t)^T b=0\}$ the maximization condition itself gives no information about $\bar u$.

We restrict our attention to extremal trajectories for which $\bar\lambda_0=1$ and for which the corresponding extremal control takes values in $\{-1,1\}$. Then $(\ref{lineardynamics})$ and $(\ref{PMP})$ lead to the differential inclusion
\begin{equation*}
    (\dot x,\dot\lambda)\in F(x,\lambda)
\end{equation*} on the space generated by the state and adjoint variables, where $F$ is defined in $(\ref{mapping})$ and, for each $(x,\lambda) \in \R^{2n}$,
$$h(x,\lambda):=\lambda^T b,$$
\begin{equation*}
    f(x,\lambda) := \begin{pmatrix} Ax+b \\ -A^T \lambda + 2 Q x \end{pmatrix} \quad\text{ and }\quad g(x,\lambda) := \begin{pmatrix} Ax-b \\ -A^T \lambda + 2 Q x \end{pmatrix}.
\end{equation*}

Note that, if $(\bar x,\bar\lambda)$ is a trajectory of $F$ on $[0,\bar T]$ satisfying $(\ref{HamiltonianZero})$ (in fact, it suffices to check that $(\ref{HamiltonianZero})$ holds only at $t=0$; see ???), then $\bar x$ is an extremal trajectory, hence a candidate for an optimal solution of [LQ] provided that it also satisfies the boundary conditions $\bar x(0)=x_0$ and $\bar x(\bar T)=x_1$.

As a consequence of Theorem \ref{main}, we obtain the following sufficient condition, expressed in terms of the matrices $Q$ and $A$ and the vector $b$, for existence of trajectories of $F$.

\begin{corollary}
    Let $C_0:=\mathbf{0}$ and $C_j:= \sum_{i=0}^{j-1} (-1)^i (A^i)^T Q A^{j-1-i}$ for each $j \ge 1$. Let $p_0 \ge 0$ be such that $QA^p b = \mathbf{0}$ for each $p \in \{0,\dots,p_0-1\}$ and $QA^{p_0} b \ne \mathbf{0}$. For $\sigma \in \{-1,1\}$, let $B_{\sigma}$ be the set of solutions $(x,\lambda) \in \R^{2n}$ of the linear system
    $$\left|
    \begin{aligned}
        (-1)^j\lambda^T A^jb+2b^T C_jx &= 0, \quad j\in\{0,\dots,2p_0+1\},\\
        \lambda^T A^{2p_0+2}b+2b^T C_{2p_0+2}x &= 2\sigma\,(A^{p_0}b)^T QA^{p_0}b.
    \end{aligned}
    \right.$$
    Then for each $(x_*,\lambda_*) \in \R^{2n} \setminus (B_{-1} \cup B_1)$ there exists a trajectory of $F$ starting from $(x_*,\lambda_*)$.
\end{corollary}
\begin{proof}
    We show that, for $k_0=2p_0+2$, \ref{A0}, \ref{A1}, and \ref{A2} hold on the open set $\R^{2n} \setminus (B_{-1} \cup B_1)$, and then we apply Theorem \ref{main}. Since $B_{-1}$ and $B_1$ are affine subspaces of $\R^{2n}$, hence closed, the set $\R^{2n} \setminus (B_{-1}\cup B_1)$ is open. Theorem \ref{main} is stated on the whole space, but its conclusion at a point depends only on $F$ in a neighborhood of that point, so it can be applied on $\R^{2n} \setminus (B_{-1} \cup B_1)$.

    We first note that \ref{A0} is satisfied on $\R^{2n}$: $Dh(x,\lambda)=(0,b^T)$ is nonzero for each $(x,\lambda) \in \R^{2n}$ because $b$ is a nonzero vector.
    
    Using induction, we next show that
    \begin{equation}\label{LQderivatives}
        L^p_f h(x,\lambda)=(-1)^p\lambda^T A^p b + 2b^T C_p x=L^p_g h(x,\lambda)
    \end{equation} for each $(x,\lambda) \in \R^{2n}$ and each $p \in \{0,\dots,2p_0+1\}$. For $p=0$, we have
    $$L^0_f h(x,\lambda)=L^0_g h(x,\lambda)=h(x,\lambda)=\lambda^T b$$ for each $(x,\lambda) \in \R^{2n}$.
    Assume that $(\ref{LQderivatives})$ holds for some $p \in \{0,\dots,2p_0\}$ and each $(x,\lambda) \in \R^{2n}$. Then $$D(L^p_f h)(x,\lambda)=(2b^T C_p, (-1)^p b^T (A^p)^T)$$ for each $(x,\lambda) \in \R^{2n}$, and thus
    \begin{align*}
        L^{p+1}_f h(x,\lambda)&=D(L^p_f h)(x,\lambda)f(x,\lambda)\\
        &=2b^T C_p(Ax+b)+(-1)^p b^T (A^p)^T(-A^T \lambda + 2 Q x)\\
        &=(-1)^{p+1} \lambda^T A^{p+1}b + 2 b^T (C_p A + (-1)^p (A^p)^T Q)x + 2 b^T C_p b
    \end{align*}
    for each $(x,\lambda) \in \R^{2n}$.
    Since $p \le 2p_0$ and $QA^mb=\mathbf{0}$ for each $m \in \{0,\dots,p_0-1\}$, it follows that $b^T C_p b = \sum_{i=0}^{p-1} (-1)^i (A^i b)^T Q A^{p-1-i} b=0$. Also, by the definition of the matrices $C_p$ and $C_{p+1}$, we have $C_{p+1}=C_p A + (-1)^p (A^p)^T Q$. Thus, for each $(x,\lambda) \in \R^{2n}$, we have $$L^{p+1}_f h(x,\lambda)=(-1)^{p+1} \lambda^T A^{p+1}b + 2 b^T C_{p+1} x.$$
    Since $b^T C_p b = 0$, the same computation gives $L^{p+1}_g h(x,\lambda)=L^{p+1}_f h(x,\lambda)$ for each $(x,\lambda) \in \R^{2n}$. Hence, $(\ref{LQderivatives})$ holds for each $p \in \{0,\dots,2p_0+1\}$ and each $(x,\lambda) \in \R^{2n}$. In particular, \ref{A1} is satisfied.

    With a similar computation as the one in the induction step above, we obtain
    \begin{align*}
        L^{2p_0+2}_f h(x,\lambda)&=(-1)^{2p_0+2}\lambda^T A^{2p_0+2} b + 2b^T C_{2p_0+2} x+2b^T C_{2p_0+1} b\\
        &=\lambda^T A^{2p_0+2} b + 2b^T C_{2p_0+2} x+ (-1)^{p_0}2(A^{p_0} b)^T Q A^{p_0} b
    \end{align*}
    for each $(x,\lambda) \in \R^{2n}$, since $b^T C_{2p_0+1} b=\sum_{i=0}^{2p_0} (-1)^i (A^i b)^T Q A^{2p_0-i} b=(-1)^{p_0} (A^{p_0} b)^T Q A^{p_0} b$.
    Analogously, we have
    $$L^{2p_0+2}_g h(x,\lambda)=\lambda^T A^{2p_0+2} b + 2b^T C_{2p_0+2} x- (-1)^{p_0} 2(A^{p_0} b)^T Q A^{p_0} b$$ for each $(x,\lambda) \in \R^{2n}$.

    Recall that $S_0=\{(x,\lambda) \in \Omega_0: L^k h(x,\lambda)=0, k \in \{1,\dots,2p_0+1\}\}$. By \eqref{LQderivatives}, $S_0=\{(x,\lambda) \in \R^{2n}:(-1)^j \lambda^T A^j b+2b^T C_j x=0 \text{ for each } j \in \{0,\dots,2p_0+1\}\}$,
    and, for each $(x,\lambda) \in S_0$, we have $L^{2p_0+2}_f h(x,\lambda)=0$ or $L^{2p_0+2}_g h(x,\lambda)=0$ if and only if $(x,\lambda) \in B_{-1} \cup B_1$. Hence, $L^{2p_0+2}_f h \ne 0$ and $L^{2p_0+2}_g h \ne 0$ at each point of $S_0 \setminus (B_{-1} \cup B_1)$. Thus, \ref{A2} holds on $\R^{2n} \setminus (B_{-1} \cup B_1)$.

    We now apply Theorem \ref{main} to $F$ on $\R^{2n} \setminus (B_{-1} \cup B_1)$.
    If $p_0=0$ (that is, $k_0=2$), we have $b^T Q b > 0$ because $Q$ is positive semi-definite and $Qb \ne \mathbf{0}$. Then $L^2_f h > L^2_g h$ on $\R^{2n}$, and thus $S_2 = \emptyset$. Hence, Theorem \ref{main}(a) implies that, for each $(x_*,\lambda_*) \in \R^{2n} \setminus (B_{-1} \cup B_1)$, there exists a trajectory of $F$ starting from $(x_*,\lambda_*)$.
    If $p_0 \ge 1$ (that is, $k_0 \ge 4$), Theorem \ref{main}(b) implies that, for each $(x_*,\lambda_*) \in \R^{2n} \setminus (B_{-1} \cup B_1)$, there exists a trajectory of $F$ starting from $(x_*,\lambda_*)$. This completes the proof.
\end{proof}

\end{document}